\documentclass[11pt]{amsart}
\usepackage[utf8]{inputenc}
\usepackage{graphicx,color}
\usepackage{amssymb}
\usepackage{amsthm}
\usepackage{amsmath}
\usepackage{amsfonts}
\usepackage{mathtools}
\usepackage{booktabs}
\usepackage{hyperref}
\usepackage[T1]{fontenc}
\usepackage{tasks}
\usepackage{lineno}
\usepackage{enumitem}
\usepackage{caption}
\usepackage{subcaption}
\usepackage{tikz}
 \usepackage{multicol}
\usetikzlibrary{shapes}

\newtheorem{innercustomthm}{Theorem}
\newenvironment{customthm}[1]
  {\renewcommand\theinnercustomthm{#1}\innercustomthm}
  {\endinnercustomthm}

\newtheorem{thm}{Theorem}[section]
\newtheorem{prop}[thm]{Proposition}
\newtheorem*{prop*}{Proposition}
\newtheorem*{thm*}{Theorem}
\newtheorem{lem}[thm]{Lemma}
\newtheorem{cor}[thm]{Corollary}

\theoremstyle{definition}
\newtheorem{dfn}[thm]{Definition}

\theoremstyle{remark}
\newtheorem{rmk}[thm]{Remark}
\newtheorem{eg}[thm]{Example}
\newtheorem{clm}[thm]{Claim}

\usepackage{verbatim}

\renewcommand{\bar}{\protect\overline}
\renewcommand{\hat}{\protect\widehat}

\newcommand{\x}{\mathbf x}

\newcommand{\reg}{\mathop{\mathrm{reg}}\nolimits}

\usepackage{tikz-cd}

\usetikzlibrary{arrows,patterns}
\pgfdeclarepatternformonly{hard horizontal lines}{\pgfpointorigin}{\pgfqpoint{100pt}{1pt}}{\pgfqpoint{100pt}{3pt}}%
{
  \pgfsetstrokeopacity{1}
  \pgfsetlinewidth{0.2pt}
  \pgfpathmoveto{\pgfqpoint{0pt}{0.2pt}}
  \pgfpathlineto{\pgfqpoint{100pt}{0.5pt}}
  \pgfusepath{stroke}
}

\pgfdeclarepatternformonly{soft crosshatch}{\pgfqpoint{-1pt}{-1pt}}{\pgfqpoint{4pt}{4pt}}{\pgfqpoint{3pt}{3pt}}%
{
  \pgfsetstrokeopacity{1}
  \pgfsetlinewidth{0.4pt}
  \pgfpathmoveto{\pgfqpoint{3.1pt}{0pt}}
  \pgfpathlineto{\pgfqpoint{0pt}{3.1pt}}
  \pgfpathmoveto{\pgfqpoint{0pt}{0pt}}
  \pgfpathlineto{\pgfqpoint{3.1pt}{3.1pt}}
  \pgfusepath{stroke}
}

\NewDocumentCommand{\statcirc}{ O{#2} m }{%
    \begin{tikzpicture}
    \fill[#2] (0,0) circle (1.0ex); 
    \fill[#1] (0,0) -- (180:1ex) arc (180:0:1ex) -- cycle; 
    \end{tikzpicture}
}

\NewDocumentCommand{\statcircc}{ O{#2} m }{%
    \begin{tikzpicture}
    \fill[#2] (0,0) circle (1.0ex); 
    \fill[#1] (0,0) -- (90:1ex) arc (90:270:1ex) -- cycle;
    \end{tikzpicture}
}

\title[Generic Lines in Projective Space and the Koszul Property]{Generic Lines in Projective Space and the Koszul Property}
\author[J. Rice]{Joshua Andrew Rice}
\address{Iowa State University, Department of Mathematics, Ames, IA, USA}
\email{jar238@iastate.edu}

\begin{document}

\subjclass[2020]{Primary: 13D02, 16S37, 14N20}

\setlength{\abovedisplayskip}{3pt}
\setlength{\belowdisplayskip}{3pt}

\keywords{generic lines, Koszul algebras, free resolutions, Castelnuovo-Mumford regularity}
 
\begin{abstract}
In this paper, we study the Koszul property of the homogeneous coordinate ring of a generic collection of lines in $\mathbb{P}^n$ and the homogeneous coordinate ring of a collection of lines in general linear position in $\mathbb{P}^n.$ We show that if $\mathcal{M}$ is a collection of $m$ lines in general linear position in $\mathbb{P}^n$ with $2m \leq n+1$ and $R$ is the coordinate ring of $\mathcal{M},$ then $R$ is Koszul. Further, if $\mathcal{M}$ is a generic collection of $m$ lines in $\mathbb{P}^n$ and $R$ is the coordinate ring of $\mathcal{M}$ with $m$ even and $m +1\leq n$ or $m$ is odd and $m +2\leq n,$ then $R$ is Koszul. Lastly, we show if $\mathcal{M}$ is a generic collection of $m$ lines such that
\[ m > \frac{1}{72}\left(3(n^2+10n+13)+\sqrt{3(n-1)^3(3n+5)}\right),\]
then $R$ is not Koszul. We give a complete characterization of the Koszul property of the coordinate ring of a generic collection of lines for $n \leq 6$ or $m \leq 6$. We also determine the Castelnuovo-Mumford regularity of the coordinate ring for a generic collection of lines and the projective dimension of the coordinate ring of collection of lines in general linear position.
\end{abstract}

\maketitle

\section{Introduction}
Let $S=\mathbb{C}[x_0,\ldots,x_{n}]$ be a polynomial ring and $J$ a graded homogeneous ideal of $S.$ Following Priddy's work, we say the ring $R=S/J$ is \textit{Koszul} if the minimal graded free resolution of the field $\mathbb{C}$ over $R$ is linear \cite{P}. Koszul rings are ubiquitous in commutative algebra. For example, any polynomial ring, all quotients by quadratic monomial ideals, all quadratic complete intersections, the coordinate rings of Grassmannians in their Pl\"{u}cker embedding, and all suitably high Veronese subrings of any standard graded algebra are all Koszul \cite{F}. Because of the ubiquity of Koszul rings, it is of interest to determine when we can guarantee a coordinate ring will be Koszul. In 1992, Kempf proved the following theorem
  \begin{thm}[Kempf, {\cite[Theorem 1]{KG}}]\label{thm:general}
Let $\mathcal{P}$ be a collection of $p$ points in $\mathbb{P}^n$ and $R$ the coordinate ring of $\mathcal{P}.$ If the points of $\mathcal{P}$ are in general linear position and $p \leq 2n,$ then $R$ is Koszul.
  \end{thm} 
In 2001, Conca, Trung, and Valla extended the theorem to a generic collection of points.
\begin{thm}[Conca, Trung, Valla, {\cite[Theorem 4.1]{CTV}}]\label{thm:generic}
Let $\mathcal{P}$ be a generic collection of $p$ points in $\mathbb{P}^n$ and $R$ the coordinate ring of $\mathcal{P}.$ Then $R$ is Koszul if and only if $p \leq 1 +n + \frac{n^2}{4}.$
\end{thm}
\indent We aim to generalize these theorems to collections of lines. In Section \ref{section:background}, we review necessary background information and results related to Koszul algebras that we use in  the other sections. In Section \ref{section:properties}, we study properties of coordinate rings of collections of lines and how they differ from coordinate rings of collections of points. In particular, we show  
\begin{customthm}{\ref{thm:reg}}
Let $\mathcal{M}$ be a generic collection of $m$ lines in $\mathbb{P}^n$ with $n \geq 3$ and $R$ the coordinate ring of $\mathcal{M}.$ Then $\mathrm{reg}_S(R)=\alpha,$ where $\alpha$ is the smallest non-negative integer such that $\binom{n+\alpha}{\alpha} \geq m(\alpha+1).$
\end{customthm}
In Section \ref{section:filtration}, we prove
\begin{customthm}{\ref{thm:koszulThm}} Let $\mathcal{M}$ be a generic collection of $m$ lines in $\mathbb{P}^n$ such that $m \geq 2$ and $R$ the coordinate ring of $\mathcal{M}.$
\begin{enumerate}[label=(\alph*)]
    \item If $m$ is even and $m+1 \leq n,$ then $R$ has a Koszul filtration. 
    \item If $m$ is odd and $m + 2 \leq n,$ then $R$ has a Koszul filtration. 
\end{enumerate}
In particular, $R$ is Koszul.
\end{customthm}
\indent Additionally, we show the coordinate ring of a generic collection of $5$ lines in $\mathbb{P}^6$ is Koszul by constructing a Koszul filtration. In Section \ref{section:negative}, we prove
\begin{customthm}{\ref{thm:notKoszul}}
Let $\mathcal{M}$ be a generic collection of $m$ lines in $\mathbb{P}^n$ and $R$ the coordinate ring of $\mathcal{M}.$ If
$$ m > \frac{1}{72}\left(3(n^2+10n+13)+\sqrt{3(n-1)^3(3n+5)}\right),$$
then $R$ is not Koszul.
\end{customthm}
\noindent Further, there is an exceptional example of a coordinate ring that is not Koszul; if $\mathcal{M}$ is a collection of $3$ lines in general linear position in $\mathbb{P}^4,$ then the coordinate ring $R$ is not Koszul.
In Section \ref{section:examples}, we exhibit a collection of lines that is not a generic collection but the lines are in general linear position, and we give two examples of coordinate rings where each define a generic collection of lines with quadratic defining ideals but for numerical reasons each coordinate ring is not Koszul. We end the document with a table summarizing the results of which coordinates rings are Koszul, which are not Koszul, and which are unknown.

\section{Background}
\label{section:background}
Let $\mathbb{P}^n$ denote $n$-dimensional projective space obtained from a $\mathbb{C}$-vector space of dimension $n+1$. A commutative Noetherian $\mathbb{C}$-algebra $R$ is said to be \textit{graded} if $R= \bigoplus_{i \in \mathbb{N}} R_i$ as an Abelian group such that for all non-negative integers $i$ and $j$ we have $R_iR_j \subseteq R_{i+j},$ and is \textit{standard graded} if $R_0= \mathbb{C}$ and $R$ is generated as a $\mathbb{C}$-algebra by a finite set of degree $1$ elements. Additionally, an $R$-module $M$ is called graded if $R$ is graded and $M$ can be written as $M= \bigoplus_{i \in \mathbb{N}} M_i$ as an Abelian group such that for all non-negative integers $i$ and $j$ we have $R_iM_j \subseteq M_{i+j}.$
 Note each summand $R_i$ and $M_i$ is a $\mathbb{C}$-vector space of finite dimension. We always assume our rings are standard graded. Let $S$ be the symmetric algebra of $R_1$ over $\mathbb{C};$ i.e. $S$ is the polynomial ring $S=\mathbb{C}[x_0,\ldots,x_n],$ where $\mathrm{dim}(R_1)=n+1$ and $x_0,\ldots,x_n$ is a $\mathbb{C}$-basis of $R_1.$ We have an induced surjection $S \rightarrow R$ of standard graded $\mathbb{C}$-algebras, and so $R \cong S/J,$ where $J$ is a homogenous ideal and the kernel of this map. We say that $J$ defines $R$ and call this ideal $J$ the \textit{defining ideal.} Denote by $\mathfrak{m}_R$ the maximal homogeneous ideal of $R.$ Except when explicitly said, all rings are graded and 
 Noetherian and all modules are finitely generated. We may view $\mathbb{C}$ as a graded $R$-module since $\mathbb{C} \cong R/\mathfrak{m}_R.$ The function $\mathrm{Hilb}_M:\mathbb{N} \rightarrow \mathbb{N}$ defined by $\mathrm{Hilb}_M(d) = \mathrm{dim}_{\mathbb{C}}(M_d)$ is called the \textit{Hilbert function} of the $R$-module $M.$ Further, there exists a unique polynomial $\mathrm{HilbP}(d)$ with rational coefficents, called the \textit{Hilbert polynomial} such that $\mathrm{HilbP}(d) = \mathrm{Hilb}(d)$ for $d \gg 0$. \par 
 \indent \textit{The minimal graded free resolution} $\textbf{F}$ of an $R$-module $M$ is an exact sequence of homomorphisms of finitely generated free $R$-modules
\[\textbf{F}: \cdots \rightarrow F_n \xrightarrow{d_n} F_{n-1} \xrightarrow{d_{n-1}} \cdots \rightarrow F_1 \xrightarrow{d_1} F_{0},\]
such that $d_{i-1}d_i = 0$ for all $i,$ $M \cong F_0/\mathrm{Im}(d_1),$ and $d_{i+1}(F_{i+1}) \subseteq (x_0,\ldots,x_n)F_i$ for all $i \geq 0.$ After choosing bases, we may represent each map in the resolution as a matrix. We can write $F_i = \bigoplus_j R(-j)^{\beta_{i,j}^R(M)},$ where $R(-j)$ denotes a rank one free module with a generator in degree $j,$ and the numbers $\beta_{i,j}^R(M)$ are called the \textit{graded Betti numbers} of $M$ and are numerical invariants of $M$. The \textit{total Betti numbers} of $M$ are defined as $\beta_i^R(M) = \sum_{j} \beta_{i,j}^R(M)$. When it is clear which module we are speaking about, we will write $\beta_{i,j}$ and $\beta_i$ to denote the graded Betti numbers and total Betti numbers, respectively. By construction, we have the equalities
\begin{align*}
    \beta_i^R(M) &= \mathrm{dim}_{\mathbb{C}}\mathrm{Tor}_i^R(M,\mathbb{C}), \\
    \beta_{i,j}^R(M) &= \mathrm{dim}_{\mathbb{C}}\mathrm{Tor}_i^R(M,\mathbb{C})_j.
\end{align*}
\indent  Two more invariants of a module are its \textit{projective dimension} and \textit{relative Castelnuovo-Mumford regularity}. These invariants are defined for an $R$-module $M$ as follows:
\[\mathrm{pdim}_R (M) = \sup\{ i\,|\, F_i \neq 0 \}= \sup \{ i \,|\, \beta_i(M) \neq 0 \}, \]
\[\reg_R (M) = \sup\{ j-i\,|\, \beta_{i,j}(M) \neq 0\}.  \]
Both invariants are interesting and measure the growth of the resolution of $M.$ For instance, if $R=S,$ then by Hilbert's Syzygy Theorem we are guaranteed that $\mathrm{pdim}_S(M) \leq n+1,$ where $n+1$ is the number of indeterminates of $S$. \par
\indent Certain invariants are related to one another. For example, if $\mathrm{pdim}_R(M)$ is finite, then the Auslander-Buchsbaum formula relates the projective dimension to the \textit{depth} of a module \cite[Theorem 15.3]{peeva2010graded}, where the depth of an  $R$-module $M$ is the length of the largest $M$-regular sequence consisting of elements of $R,$ and is denoted $\mathrm{depth}(M).$ Letting $R=S,$ the Auslander-Buchsbaum formula states that the projective dimension and depth of an $S$-module $M$ are complementary to one another:
\begin{align}\label{eq:auslanderBuchsbaum}
    \mathrm{pdim}_S(M) + \mathrm{depth}(M) = n+1.
\end{align} 
\indent The \textit{Krull dimension}, or \textit{dimension}, of a ring is the supremum of the lengths $k$ of strictly increasing chains $P_0 \subset P_1 \subset \ldots \subset P_k$ of prime ideals of $R.$ The \textit{dimension} of an $R$-module is denoted $\mathrm{dim}(M)$ and is the Krull dimension of the ring $R/I,$ where $I = \mathrm{Ann}_R(M)$ is the annihilator of $M.$ The depth and dimension of a ring have the following properties along a short exact sequence. 
\begin{prop}{\cite[Corollary 18.6]{eisenbud1995commutative}} \label{prop:depthInequality}
Let $R$ be a graded Noetherian ring and suppose that
\[ 0 \rightarrow M' \rightarrow M \rightarrow M^{''} \rightarrow 0 \]
is an exact sequence of finitely generated graded $R$-modules. Then
\begin{enumerate}[label=(\alph*)]
    \item $\mathrm{depth}(M^{'}) \geq \mathrm{min}\{ \mathrm{depth}(M), \mathrm{depth}(M^{''}) +1 \}, $\\
    \item $\mathrm{depth}(M) \geq \mathrm{min}\{ \mathrm{depth}(M^{'}), \mathrm{depth}(M^{''}) \}, $\\
    \item $\mathrm{depth}(M^{''}) \geq \mathrm{min}\{ \mathrm{depth}(M), \mathrm{depth}(M^{'}) - 1 \}, $\\
    \item $\mathrm{dim}(M)= \mathrm{max}\{ \mathrm{dim}(M^{''}),\mathrm{dim}(M^{'})\}.$
\end{enumerate}
Furthermore, $\mathrm{depth}(M) \leq \mathrm{dim}(M).$
\end{prop}

 An R-module $M$ is \textit{Cohen-Macaulay}, if  $\mathrm{depth}(M) = \mathrm{dim}(M).$ Since $R$ is a module over itself, we say $R$ is a \textit{Cohen-Macaulay ring} if it is a Cohen-Macaulay module $R$-module. Cohen-Macaulay rings have been studied extensively, and the definition is sufficiently general to allow a rich theory with a wealth of examples in algebraic geometry. This notion is a workhorse in commutative algebra, and provides very useful tools and reductions to study rings \cite{bruns1998cohen}. For example, if one has a graded Cohen-Macaulay $\mathbb{C}$-algebra, then one can take a quotient by generic linear forms to produce an Artinian ring. A reduction of this kind provides many useful tools to work with, and almost all homological invariants of the ring are preserved \cite{migliore2011minimal}. Unfortunately, we will not be able to use these tools or reductions as the coordinate ring of a generic collection of lines is almost never Cohen-Macaulay, whereas the coordinate ring of a generic collection of points is always Cohen-Macaulay.\par
\indent The \textit{absolute Castelnuovo-Mumford regularity}, or the \textit{regularity}, is denoted $\mathrm{reg}_S(M)$ and is the regularity of $M$ as an $S$-module. There is a cohomological interpretation by local duality \cite{EG}. Set $H_{\mathfrak{m}_S}^i(M)$ to be the $i^{th}$ local cohomology module with support in the graded maximal ideal of $S.$ One has $H_{\mathfrak{m}_S}^i(M) =0$ if $i < \mathrm{depth}(M)$ or $i > \mathrm{dim}(M)$ and 
\[ \mathrm{reg}_S(M) = \max \{j+i : H_{\mathfrak{m}_S}^i(M)_j  \neq 0 \}.\] 
\noindent In practice, bounding the regularity of $M$ is difficult, since it measures the largest degree of a minimal syzygy of $M$. We have tools to help the study of the regularity of an $S$-module.
\begin{prop}{\cite[Exercise 4C.2, Theorem 4.2, Corollary 4.4]{CTV}} \label{prop:regInequality}
Suppose that
\[ 0 \rightarrow M' \rightarrow M \rightarrow M^{''} \rightarrow 0 \]
is an exact sequence of finitely generated graded $S$-modules. Then
\begin{enumerate}[label=(\alph*)]
    \item $\mathrm{reg}_S(M^{'}) \leq \mathrm{max}\{ \mathrm{reg}_S(M), \mathrm{reg}_S(M^{''}) +1 \}, $\\
    \item $\mathrm{reg}_S(M) \leq \mathrm{max}\{ \mathrm{reg}_S(M^{'}), \mathrm{reg}_S(M^{''}) \}, $\\
    \item $\mathrm{reg}_S(M^{''}) \leq \mathrm{max}\{ \mathrm{reg}_S(M), \mathrm{reg}_S(M^{'}) - 1 \}, $\\
\end{enumerate}
and if $d_0 = \mathrm{min}\{d \, | \, \mathrm{Hilb}(d) = \mathrm{HilbP}(d) \},$ then $\mathrm{reg}(M) \geq d_0.$ Furthermore, if $M$ is Cohen-Macaulay, then $\mathrm{reg}_S(M) = d_0.$ If $M$ has finite length, then $\mathrm{reg}_S(M) = \mathrm{max}\{ d : M_d \neq 0 \}.$ 
\end{prop}
To study these invariants, we place the graded Betti numbers of a module $M$ into a table, called \textit{the Betti table}
\[\begin{tikzpicture}

\draw[thick,double](1,-4)--(1,0);
\draw[thick,double](0,-1)--(7,-1);

\draw[step=1cm,black,very thin] (7,-4) grid (0,0);
\foreach \x in {0,1,2,3,4,5}
\node at (1+\x+.5,-.5) {$\x$};

\node at (.5,-1-0-.5) {$0$};
\node at (.5,-1-1-.5) {$1$};
\node at (.5,-1-2-.5) {$\vdots$};

\node at (.5,-.5) {$M$};

\node at (1.5,-1.5) {$\beta_{0,0}$};
\node at (2.5,-1.5) {$\beta_{1,1}$};
\node at (3.5,-1.5) {$\beta_{2,2}$};
\node at (4.5,-1.5) {$\beta_{3,3}$};
\node at (5.5,-1.5) {$\beta_{4,4}$};
\node at (6.5,-1.5) {$\cdots$};

\node at (1.5,-2.5) {$\beta_{0,1}$};
\node at (2.5,-2.5) {$\beta_{1,2}$};
\node at (3.5,-2.5) {$\beta_{2,3}$};
\node at (4.5,-2.5) {$\beta_{3,4}$};
\node at (5.5,-2.5) {$\beta_{4,5}$};
\node at (6.5,-2.5) {$\cdots$};

\node at (1.5,-3.5) {$\vdots$};
\node at (2.5,-3.5) {$\vdots$};
\node at (3.5,-3.5) {$\vdots$};
\node at (4.5,-3.5) {$\vdots$};
\node at (5.5,-3.5) {$\vdots$};
\node at (6.5,-3.5) {$\ddots$};
\end{tikzpicture}\]
\noindent The Betti table allows us to determine certain invariants easier; e.g., the projective dimension is the length of the table and the regularity is the height of the table. 

Denote by $H_M(t)$ and $P_{M}^{R}(t)$ respectively the \textit{Hilbert series of} $M$ and the \textit{Poincar\'{e} series of an $R$-module} $M$:
\[ H_M(t) = \sum_{i \geq 0} \mathrm{Hilb}_M(i) t^i \]
and 
\[ P_{M}^R(t) = \sum_{i \geq 0} \beta_i^R(M) t^i. \]
It is worth observing that since $M$ is finitely generated by homogenous elements of positive degree, the Hilbert series of $M$ is a rational function. A short exact sequence of modules has a property we use extensively in this paper. If we have a short exact sequence of graded $S$-modules
\[ 0 \longrightarrow A \longrightarrow B \longrightarrow C \longrightarrow 0, \]
then
\[ \mathrm{H}_B(t) = \mathrm{H}_A(t) + \mathrm{H}_C(t).\]
Whenever we use this property, we will refer to it as the additivity property of the Hilbert series.\par
\indent A standard graded $\mathbb{C}$-algebra $R$ is \textit{Koszul} if $\mathbb{C}$ has a linear $R$-free resolution; that is, $\beta_{i,j}^R(\mathbb{C}) =0$ for $i \neq j$. Koszul algebras possess remarkable homological properties.  For example
\begin{thm}[Avramov, Eisenbud, and Peeva, {\cite[Theorem 1]{AE}  \cite[Theorem 2]{AP}}]
The following are equivalent:
\begin{enumerate}
    \item Every finitely generated $R$-module has finite regularity.
    \item The residue field has finite regularity.
    \item $R$ is Koszul.
\end{enumerate}
\end{thm}
\noindent Koszul rings possess other interesting properties as well. Fr\"{o}berg showed in \cite{froberg1999koszul} that $R$ is Koszul if and only if $H_R(t)$ and the $P_\mathbb{C}^R(t)$ have the following relationship
\begin{align}\label{eqn:froberg}
    P_{\mathbb{C}}^R(t) H_R(-t) = 1.
\end{align} 
In general, the Poincar\'{e} series of $\mathbb{C}$ as an $R$-module can be irrational \cite{Anick}, but if $R$ is Koszul, then Equation (\ref{eqn:froberg}) tells us the Poincar\'{e} series is always rational. So a necessary condition for a coordinate ring $R$ to be Koszul is $P_{\mathbb{C}}^R(t) = \frac{1}{H_R(-t)}$ must have non-negative coefficients in its Maclaurin series. Another necessary condition is that if $R$ is Koszul, then the defining ideal has a minimal generating set of forms of degree at most $2$. This is easy to see since 
\[ \beta_{2,j}^R(\mathbb{C}) = \begin{cases} 
       \beta_{1,j}^S(R) & \text{if }j \neq 2 \\[.5em]
      \beta_{1,2}^S(R) + \binom{n+1}{2} & \text{if } j=2, \\
   \end{cases}
\]
\cite[Remark 1.10]{C}. Unfortunately, the converse does not hold, but Fr\"{o}berg showed that if the defining ideal is generated by monomials of degree at most $2,$ then $R$ is Koszul.
\begin{thm}[Fr\"{o}berg, \cite{F}]\label{thm:Frobergs}
If $R=S/J$ and $J$ is a monomial ideal with each monomial having degree at most $2$, then $R$ is Koszul.
\end{thm}\par
\indent More generally, if $J$ has a Gr\"{o}bner basis of quadrics in some term order, then $R$ is Koszul. If such a basis exists, we say that $R$ is \textit{G-quadratic}. More generally, $R$ is \textit{LG-quadratic} if there is a G-quadratic ring $A$ and a regular sequence of linear forms $l_1,\ldots,l_r$ such that $R \cong A/(l_1,\ldots,l_r).$ It is worth noting that every G-quadratic ring is LG-quadratic, and every LG-quadratic ring is Koszul and that all of these implications are strict \cite{C}. We briefly discuss in Section \ref{section:examples} if coordinate rings of generic collections of lines are G-quadratic or LG-quadratic. \par
\indent We now define a very useful tool in proving rings are Koszul.
\begin{dfn}\label{dfn:koszulFiltration}
Let $R$ be a standard graded $\mathbb{C}$-algebra. A family $\mathcal{F}$ of ideals is said to be a \textit{Koszul filtration} of $R$ if
\begin{enumerate}[label=(\alph*)]
    \item Every ideal $I \in\mathcal{F}$ is generated by linear forms,
    \item The ideal $0$ and the maximal homogeneous ideal $\mathfrak{m}_R$ of $R$ belong to $\mathcal{F},$
    \item For every ideal $I \in \mathcal{F}$ different from $0,$ there exists an ideal $K \in \mathcal{F}$ such that $K \subset I, I/K$ is cyclic, and $K : I \in \mathcal{F}.$
\end{enumerate}

\end{dfn}
Conca, Trung, and Valla showed in \cite{CTV} that if $R$ has a Koszul filtration, then $R$ is Koszul. In fact a stronger statement is true.
\begin{prop}[{\cite[Proposition 1.2]{CTV}}]Let $\mathcal{F}$ be a Koszul Filtration of $R.$ Then $\mathrm{Tor}_i^R(R/J,\mathbb{C})_j=0$ for all $i \neq j$ and for all $J \in \mathcal{F}.$ In particular, $R$  is Koszul. 
\end{prop}

\indent Conca, Trung, and Valla construct a Koszul filtration to show certain sets of points in general linear position are Koszul in \cite{CTV}. Since we aim to generalize Theorems \ref{thm:general} and \ref{thm:generic} to collections of lines, we must define what it means for a collection of lines to be generic and what it means for a collection of lines to be in general linear position.
\begin{dfn}
Let $\mathcal{P}$ be a collection of $p$ points in $\mathbb{P}^n$ and $\mathcal{M}$ be a collection of $m$ lines in $\mathbb{P}^n.$ The points of $\mathcal{P}$ are in \textit{general linear position} if any $s$ points span a $\mathbb{P}^{r},$ where $ r = \mathrm{min} \{s-1,n\}.$ Similarly, the lines of $\mathcal{M}$ are in \textit{general linear position} if any $s$ lines span a $\mathbb{P}^{r},$ where $ r = \mathrm{min} \{2s-1,n\}.$ A collection of points in $\mathbb{P}^n$ is a \textit{generic collection} if every linear form in the defining ideal of each point has algebraically independent coefficients over $\mathbb{Q}$. Similarly, we say a collection of lines is a \textit{generic collection} if every linear form in the defining ideal of each line has algebraically independent coefficients over $\mathbb{Q}$.
\end{dfn}
We can interpret this definition as saying a generic collection is sufficiently random. As one should suspect, a generic collection of lines is in general linear position and this containement is strict. For an example demonstrating this see Section \ref{section:examples}. We end this section with a remark about collections of points and collections of lines that we use extensively; for ease of reference we include in the remark the fact that a generic collection of lines is in general linear position.
\begin{rmk}\label{rmk:linesIntersection}
Suppose $\mathcal{P}$ is a collection of $p$ points in general linear position in $\mathbb{P}^n$ and $\mathcal{M}$ is a collection of $m$ lines in general linear position in $\mathbb{P}^n.$ The defining ideal for each point is minimally generated by $n$ linear forms and the defining ideal for each line is minimally generated by $n-1$ linear forms. We can see this because a point is an intersection of $n$ hyperplanes and a line is an intersection of $n-1$ hyperplanes. Also, if $K$ is the defining ideal for $\mathcal{P}$ and $J$ is the defining ideal for $\mathcal{M},$ then $\mathrm{dim}_{\mathbb{C}}(K_1)=n+1-p$ and $\mathrm{dim}_{\mathbb{C}}(J_1)=n+1-2m,$ provided either quantity is non-zero. Generic collections of lines are in general linear position but the converse is not true; see Example 6.1.
\end{rmk}
\section{Properties of Coordinate Rings of Lines}
\label{section:properties}
This section aims to establish properties for the coordinate rings of generic collections of lines and collections of lines in general linear position and compare them to the coordinate rings of generic collections of points and collections of points in general linear position. We will see that the significant difference between the two coordinate rings is that the coordinate ring $R$ of a collection of lines in at least general linear position is never Cohen-Macaulay, unless $R$ is the coordinate ring of a single line, while the coordinate rings of points in general linear position are always Cohen-Macaulay. The lack of the Cohen-Macaulay property presents difficulty since many techniques are not available to us, such as Artinian reductions.
\begin{prop}\label{prop:properties}
Let $\mathcal{M}$ be a collection of lines in general linear position in $\mathbb{P}^n$ with $n \geq 3,$ and $R$ the coordinate of $\mathcal{M}.$ If $|\mathcal{M}|=1$, then $\mathrm{pdim}_S(R)=n-1, $ $\mathrm{depth}(R)=2,$ and $\mathrm{dim}(R)=2;$ if $|\mathcal{M}| \geq 2,$ then $\mathrm{pdim}_S(R)=n,$ $\mathrm{depth}(R)=1,$ and $\mathrm{dim}(R)=2$. In particular, $R$ is Cohen-Macaulay if and only if $|\mathcal{M}|=1.$
\end{prop}
\begin{proof}
We prove the claim by induction on $|\mathcal{M}|$. Let $m=|\mathcal{M}|$ and let $J$ be the defining ideal of $\mathcal{M}$. If $m=1,$ then by Remark \ref{rmk:linesIntersection} the ideal $J$ is minimally generated by $n-1$ linear forms. So, $R$ is isomorphic to a polynomial ring in two indeterminates. Now, suppose that $m \geq 2,$ and write $J = K \cap I,$ where $K$ is the defining ideal for $m-1$ lines and $I$ is the defining ideal for the remaining single line. By induction, $\mathrm{depth}(S/K) \leq 2$ and  $\mathrm{dim}(S/K)=\mathrm{dim}(S/I)=2.$  Furthermore, $S/(K+I)$ is Artinian, since the variety $K$ defines intersects trivially with the variety $I$ defines. Hence, $\mathrm{dim}(S/(I+K)) =0.$ So, by Proposition \ref{prop:depthInequality} the $\mathrm{depth}(S/(I+K))=0.$ \par 
\indent Using the short exact sequence 
$$\begin{tikzcd}
0 \arrow[r] & S/J \arrow[r] & S/K \oplus S/I \arrow[r] & S/(K+I) \arrow[r] & 0,
\end{tikzcd} $$
and Proposition $\ref{prop:depthInequality},$ we have two inequalities 
\[ \mathrm{min}\left\{ \mathrm{depth}(S/K \oplus S/I),\mathrm{depth}(S/(I+K))+1  \right\} \leq \mathrm{depth}(S/J) ,\]
and
\[
    \mathrm{min}\left\{ \mathrm{depth}(S/K \oplus S/I), \mathrm{depth}(S/J) - 1 \right\} \leq \mathrm{depth}(S/(I+K)).\]
Regardless if $\mathrm{depth}(S/K)$ is $1$ or $2,$ our two inequalities yield $\mathrm{depth}(S/J)=1.$ By the Auslander–Buchsbaum formula, we have $\mathrm{pdim}_S(S/J)=n.$ Lastly, Proposition $\ref{prop:depthInequality},$ yields $\mathrm{dim}(S/J) =2.$ 
\end{proof}
\begin{rmk}
We would like to note that when $n=2,$ $R$ is a hypersurface and so $\mathrm{pdim}_S(R)=1,$ $\mathrm{depth}(R)=2,$ and $\mathrm{dim}(R)=2.$ Thus, we restrict our attention to the case $n \geq 3$. Furthermore, an identical proof shows that if $\mathcal{P}$ is a collection of points in general linear position in $\mathbb{P}^n$ and $R$ is the coordinate ring of $\mathcal{P},$ then $\mathrm{pdim}_S(R)=n,$ $\mathrm{depth}(R)=1,$ and $\mathrm{dim}(R)=1.$ Hence, $R$ is Cohen-Macaulay.
\end{rmk}
In \cite{CTV}, Conca, Trung, and Valla used the Hilbert function of points in $\mathbb{P}^n$ in general linear position to prove the corresponding coordinate ring is Koszul, provided the number of points is at most $n+1.$ There is a generalization for the Hilbert function to a generic collection of points.

\begin{thm}[\cite{CCVG}]\label{thm:hsForPoints}
Let $\mathcal{P}$ be a generic collection of $p$ points in $\mathbb{P}^n$ and $R$ the coordinate ring of $\mathcal{P}.$ The Hilbert function of $R$ is
\begin{align*}
    \mathrm{Hilb}_R(d) = \mathrm{min} \left\{ \binom{n+d}{d}, p \right\}.
\end{align*} 
In particular, if $p \leq n+1$, then
\begin{align*}
    H_R (t)= \frac{(p-1)t+1}{1-t}.
\end{align*} 
\end{thm}
Since we aim to generalize Theorems \ref{thm:general} and \ref{thm:generic}, we would like to know the Hilbert series of the coordinate ring of a generic collection of lines. The famous Hartshorne-Hirschowitz Theorem provides an answer.
\begin{thm}[Hartshorne-Hirschowitz, \cite{hartshorne1982droites}, 1983]\label{thm:hh}
Let $\mathcal{M}$ be a generic collection of $m$ lines in $\mathbb{P}^n$ and $R$ the coordinate ring of $\mathcal{M}.$ The Hilbert function of $R$ is
$$ \mathrm{Hilb}_R(d) = \mathrm{min} \Bigg\{\binom{n+d}{d}, m(d+1) \Bigg\}.$$
\end{thm}
\noindent This theorem is very difficult to prove. One could ask if any generalization holds for planes, and unfortunately, this is not known and is an open problem. Interestingly, this theorem allows us to determine the regularity for the coordinate ring $R$ of a generic collection of lines. 
\begin{thm}\label{thm:reg}
Let $\mathcal{M}$ be a generic collection of $m$ lines in $\mathbb{P}^n$ with $n \geq 3$ and $R$ the coordinate ring of $\mathcal{M}.$ Then $\mathrm{reg}_S(R)=\alpha,$ where $\alpha$ is the smallest non-negative integer satisfying $\binom{n+\alpha}{\alpha} \geq m(\alpha+1).$
\end{thm}
\begin{proof}
If $m=1,$ then by Remark \ref{rmk:linesIntersection} and a change of basis we can write the defining ideal as $J=(x_0,\ldots,x_{n-2}).$ The coordinate ring $R$ is minimally resolved by the Koszul complex on $x_0,\ldots,x_{n-2}$. So, $\mathrm{reg}_S(R)=0,$ and this satisfies the inequality. Suppose that $m \geq 2$ and let $\alpha$ be the smallest non-negative integer satisfying $\binom{n+\alpha}{\alpha} \geq m(\alpha+1).$ By Theorem \ref{thm:hh} and Proposition \ref{prop:regInequality}, $\mathrm{reg}_S(R) \geq \alpha.$ \par 
\indent  We show the reverse inequality by induction on $m.$ Let $J$ be the defining ideal for the collection $\mathcal{M}.$ Note, removing a line from a generic collection of lines maintains the generic property for the new collection. Let $K$ be the defining ideal for $m-1$ lines and $I$ the defining ideal for the remaining line such that $J = K \cap I.$ By induction $\mathrm{reg}_S(S/K)=\beta,$ and $\beta$ is the smallest non-negative integer satisfying the inequality $\binom{n+\beta}{\beta} \geq (m-1)(\beta+1).$ \par 
\indent Now, we claim that $\mathrm{reg}_S(S/K) = \beta \in \{ \alpha, \alpha -1 \}.$ To prove this we need two inequalities: $ m-2 \geq \beta$ and  $\binom{n+\beta}{\beta+1} \geq n(m-1).$ We have the first inequality since
\begin{align*}
    \binom{n+m-2}{m-2}-(m-1)(m-2+1) &=  \frac{(n+m-2)!}{n!(m-2)!} - (m-1)^2 \\
    &= \frac{(m+1)!}{3!(m-2)!} - (m-1)^2   \\
    &= \frac{(m-3)(m-2)(m-1)}{3!} \\
    &\geq 0. 
    \end{align*}
Thus, $m-2 \geq \beta.$ We have the second inequality, since by assumption
\begin{align*}
    \binom{n+\beta}{\beta} &\geq (m-1)(\beta+1),
    \end{align*}
and rearranging terms gives
    \begin{align*}
    \binom{n+\beta}{\beta+1} &\geq n(m-1). 
    \end{align*}
These inequalities together yield the following
\begin{align*}
    \binom{n+\beta+1}{\beta+1} &= \binom{n+\beta}{\beta} + \binom{n+\beta}{\beta+1} \\
                               &\geq (m-1)(\beta+1) + n(m-1) \\
                                &= (m-1)(\beta+1) + m+(m-1)(n-1) - 1  \\
                                &\geq (m-1)(\beta+1) + m +(m-1)2 -1  \\
                                &\geq (m-1)(\beta+1) + m +\beta  + 1 \\
                                &= m (\beta+2). 
    \end{align*} 
Hence, $ \beta + 1 \geq \alpha.$ Furthermore, the inequality
\begin{equation*}
     \binom{n+\beta-1}{\beta-1} < (m-1)(\beta-1+1) \leq m(\beta)      
     \end{equation*}
     implies that $ \alpha \geq \beta.$ So, $\mathrm{reg}_S(S/K) = \beta \in \{ \alpha, \alpha -1 \}.$
     \par 
\indent Consider the short exact sequence
\[ 0  \longrightarrow S/J \longrightarrow    S/K \oplus S/I  \longrightarrow  S/(K+I)  \longrightarrow  0.\]
If $\beta = \alpha,$ then Theorem \ref{thm:hh} and the additive property of the Hilbert series yields the following 
\begin{align*}
    H_{S/(K+I)}(t) &= \left( H_{S/K}(t) + H_{S/I}(t)\right) - H_{S/J}(t) \\
    \end{align*}
\begin{align*}  
    &= \left(\sum\limits_{k=0}^{\alpha-1} \binom{n+k}{k} t^k +  \sum\limits_{k=\alpha}^{\infty} (m-1)(k+1)t^k + \sum\limits_{k=0}^{\infty} (k+1)t^k \right)\\
    &\hspace{1cm}- \sum\limits_{k=0}^{\alpha-1} \binom{n+k}{k} t^k -  \sum\limits_{k=\alpha}^{\infty} m(k+1)t^k \\
    &= \sum_{k=0}^{\alpha-1}(k+1)t^k.
\end{align*}
and similarly if $\beta = \alpha-1,$ then
\begin{align*}
    H_{S/(K+I)}(t) &= \left( H_{S/K}(t) + H_{S/I}(t)\right) - H_{S/J}(t) \\
    &= \left(\sum\limits_{k=0}^{\alpha-2} \binom{n+k}{k} t^k +  \sum\limits_{k=\alpha-1}^{\infty} (m-1)(k+1)t^k + \sum\limits_{k=0}^{\infty} (k+1)t^k \right)\\
    &\hspace{1cm}- \sum\limits_{k=0}^{\alpha-1} \binom{n+k}{k} t^k -  \sum\limits_{k=\alpha}^{\infty} m(k+1)t^k \\
    &= \sum_{k=0}^{\alpha-2}(k+1)t^k + \left( m\alpha - \binom{n+\alpha-1}{\alpha-1} \right)t^{\alpha-1}.
\end{align*}
Note that $m\alpha - \binom{n+\alpha-1}{\alpha-1}$ is positive since $\alpha$ is the smallest non-negative integer such that $\binom{n+\alpha}{\alpha} \geq m(\alpha +1)$. So, $S/(K+I)$ is Artinian. By Proposition \ref{prop:regInequality}, $\mathrm{reg}_S(S/(K+I))=\alpha-1.$ Since $\mathrm{reg}_S(S/K)=\alpha$ or $\mathrm{reg}_S(S/K)=\alpha-1$ and $\mathrm{reg}_S(S/I)=0,$ then $\mathrm{reg}_S(S/J) \leq \alpha$. Thus, $\mathrm{reg}_S(R)=\alpha.$ 
\end{proof}
\begin{rmk}
By Proposition \ref{prop:properties}, the coordinate ring $R$ for a generic collection of lines is not Cohen-Macaulay, but $\mathrm{reg}_S(R) = \alpha,$ where $\alpha$ is precisely the smallest non-negative integer where $\mathrm{Hilb}(d)=\mathrm{HilbP}(d)$ for $d \geq \alpha.$ By Proposition \ref{prop:regInequality}, if a ring is Cohen-Macaulay then the regularity is precisely this number. So, even though we are not Cohen-Macaulay, we do not lose everything in generalizing these theorems. 
\end{rmk}
\noindent Compare the previous result with the following general regularity bound for intersections of ideals generated by linear forms.
\begin{thm}[Derksen, Sidman, {\cite[Theorem 2.1]{DS}}]\label{thm:ds}
If $J = \bigcap_{i=1}^j I_i$ is an ideal of $S,$ where each $I_i$ is an ideal generated by linear forms, then $\mathrm{reg}_S(S/J) \leq j.$
\end{thm}
The assumption that $R$ is a coordinate ring of a generic collection of lines tells us the regularity exactly, which is much smaller then the Derksen-Sidman bound for a fixed $n$. By way of comparison we compute the following estimate.
\begin{cor}\label{cor:reg}
Let $\mathcal{M}$ be a generic collection of $m$ lines in $\mathbb{P}^n$ with $n \geq 3$ and $R$ the coordinate ring of $\mathcal{M}.$ Then 
\[\mathrm{reg}_S(R) \leq \left \lceil \sqrt[n-1]{n!} \left( \sqrt[n-1]{m} -1\right)\right \rceil.\]
\end{cor}
\begin{proof}
Let $p(x) = (x+n)\cdots(x+2) - n!m.$ The polynomial $p(x)$ has a unique positive root by the Intermediate Value Theorem, since the $(x+n)\cdots(x+2)$ is increasing on the non-negative real numbers. Let $a$ be this positive root, and observe that the smallest non-negative integer $\alpha$ satisfying the inequality $\binom{n+\alpha}{\alpha} \geq m(\alpha+1)$ is precisely the ceiling of the root $a.$ \par 
\indent We now use an inequality of Minkowski \cite[Equation (1.5)]{frenkel2013minkowskis}. If $x_k$ and $y_k$ are positive for each $k,$ then 
\[\sqrt[n-1]{\prod_{k=1}^{n-1} (x_k + y_k)} \geq \sqrt[n-1]{\prod_{k=1}^{n-1} x_k} + \sqrt[n-1]{\prod_{k=1}^{n-1}  y_k}. \]
Thus,
\begin{align*}
    \sqrt[n-1]{n!m} &= \sqrt[n-1]{ (a+n)\cdots(a+2)} \\
                                     &\geq a + \sqrt[n-1]{n!} \\
                                         \end{align*}
Therefore,
\[  \sqrt[n-1]{n!m} - \sqrt[n-1]{n!} \geq a. \]                                
Taking ceilings gives the inequality. 
\end{proof}
We would like to note that $\mathrm{reg}_S(R)$ is roughly asymptotic to the upper bound. Proposition \ref{prop:properties} and Theorem \ref{thm:reg} tell us the coordinate ring $R$ of a non-trivial generic collection of lines in $\mathbb{P}^n$ is not Cohen-Macaulay, $\mathrm{pdim}_S(R)=n,$ and the regularity is the smallest non-negative integer $\alpha$ satisfying $\binom{n+\alpha}{\alpha} \geq m(\alpha+1)$. So, the resolution of $R$ is well-behaved, in the sense that if $n$ is fixed and we allow $m$ to vary we may expect the regularity to be low compared to the number of lines in our collection.

\section{Koszul Filtration for a collection of lines}\label{Filtration}
\label{section:filtration}
In this section we determine when a generic collection of lines, or a collection of lines in general linear position, will yield a Koszul coordinate ring. To this end, most of the work will be in constructing a Koszul filtration in the coordinate ring of a generic collection of lines.

\begin{prop}\label{prop:generalLines}
 Let $\mathcal{M}$ be a collection of $m$ lines in general linear position in $\mathbb{P}^n$, with $n\geq 3$ and $R$ the coordinate ring of $\mathcal{M}.$ If $n+1 \geq 2m,$ then after a change of basis the defining ideal is minimally generated by monomials of degree at most $2.$ Thus, $R$ is Koszul.
\end{prop}
\begin{proof}
We will only prove the case when $m$ is even, since the case when $m$ is odd is identical. Furthermore, we use $\hat{\cdot}$ to denote a term removed from a sequence. Let $R$ be the coordinate ring of $\mathcal{M}$ with defining ideal $J$ and suppose $m=2k,$ for some $k.$ Through a change of basis and Remark \ref{rmk:linesIntersection} we may assume the defining ideal for each line has the following form
\begin{align*}
L_{i} &= (x_0,\ldots,\hat{x}_{n-2i+1} ,\hat{x}_{n-2i+2},\ldots,x_{n-1},x_{n}),
\end{align*}
for $i = 1,\ldots,2k.$ Since every $L_i$ is monomial, so is $J$. Furthermore, since $n+1 \geq 4k,$ Proposition \ref{thm:reg} implies $\mathrm{reg}_S(R) \leq 1.$ Thus, $J$ is generated  by monomials of degree at most $2.$ Theorem \ref{thm:Frobergs} guarantees $R$ is Koszul.
\end{proof}
\indent Unfortunately, the simplicity of the previous proof does not carry over for larger generic collections of lines. We need a lemma.
\begin{lem}\label{lem:hilbertSeries}
Let $\mathcal{M}$ be a generic collection of $m$ lines in $\mathbb{P}^n$ and $R$ the coordinate ring of $\mathcal{M}.$ If $\mathrm{reg}_S(R) = 1$, then the Hilbert series of $R$ is
\[ H_{S/J}(t) = \frac{(1-m)t^2+(m-2)t+1}{(1-t)^2}.\]
If $\mathrm{reg}_S(R) = 2$, then the Hilbert series of $R$ is
  \[   H_{S/J}(t) = \frac{(1+n-2m)t^3+(3m-2n-1)t^2+(n-1)t+1}{(1-t)^2}. \]
\end{lem}
\begin{proof}
By Theorem \ref{thm:reg}, the regularity is the smallest non-negative integer $\alpha$ satisfying $\binom{n+\alpha}{\alpha} \geq  m(\alpha+1).$  Suppose $\mathrm{reg}_S(R)=1$. By Theorem \ref{thm:hh}, the Hilbert series for $R$ is
\begin{align*}
    H_{R}(t) &= 1 + 2mt + 3mt^2 +4mt^3 + \cdots \\
               &= 1  - m\left( \frac{t(t-2)}{(1-t)^2}\right) \\
               &= \frac{t^2-2t+1-mt^2+2mt}{(1-t)^2}\\
               &= \frac{(1-m)t^2+2(m-1)t+1}{(1-t)^2}.
               \end{align*}
Now, suppose $\mathrm{reg}_S(R)=2$. By Theorem \ref{thm:hh}, the Hilbert series for $R$ is
               \begin{align*}
    H_{R}(t) &= 1  +(n+1)t + 3mt^2 +4mt^3 + \cdots \\
               &= 1+(n+1)t-m\left( \frac{t^2(2t-3)}{(1-t)^2} \right)\\
               &= \frac{(n+1)t^3-(2n+1)t^2+(n-1)t+1-2mt^3+3mt^2}{(1-t)^2}   \\
               &= \frac{(n+1-2m)t^3+(3m-2n-1)t^2+(n-1)t+1}{(1-t)^2}.
\end{align*}
\end{proof}

We can now construct a Koszul filtration for the coordinate ring of certain larger generic collections of lines.
\begin{thm}\label{thm:koszulThm} Let $\mathcal{M}$ be a generic collection of $m$ lines in $\mathbb{P}^n$ such that $n\geq 3$ and $m \geq 3$  and $R$ the coordinate ring of $\mathcal{M}.$
\begin{enumerate}[label=(\alph*)]
    \item If $m$ is even and $m+1 \leq n,$ then $R$ has a Koszul filtration. 
    \item If $m$ is odd and $m + 2 \leq n,$ then $R$ has a Koszul filtration. 
\end{enumerate}
In particular, $R$ is Koszul.
\end{thm}
\begin{proof}
We only prove $(a)$ due to the length of the proof and note that $(b)$ is done identically except for the Hilbert series computations. In both cases we may assume that $n \leq 2(m-1),$ otherwise Proposition \ref{prop:generalLines} and Remark \ref{rmk:linesIntersection} prove the claim. By Remark \ref{rmk:linesIntersection} and a change of basis, we may assume the defining ideals for our $m$ lines have the following form
\begin{align*}
 L_1 &= (x_0,\ldots,x_{n-4},x_{n-3},x_{n-2})\\
 L_2 &= (x_0,\ldots,x_{n-4},x_{n-1},x_{n})\\
&\hspace{2.5cm}\vdots\\
   L_i &= (x_0,\ldots,\hat{x}_{n-2i+1},\hat{x}_{n-2i+2},\ldots,x_{n})\\
   &\hspace{2.5cm}\vdots\\
L_k &= (x_0,\ldots,\hat{x}_{n-2k+1},\hat{x}_{n-2k+2},\ldots,x_{n})\\
L_{k+1} &= (l_0,\ldots,l_{n-4},l_{n-3},l_{n-2})\\
 L_{k+2} &= (l_0,\ldots,l_{n-4},l_{n-1},l_{n})\\
&\hspace{2.5cm}\vdots\\
  L_{k+i} &= (l_0,\ldots,\hat{l}_{n-2i+1} ,\hat{l}_{n-2i+2},\ldots,l_{n})\\
&\hspace{2.5cm}\vdots\\
L_{2k} &= (l_0,\ldots,\hat{l}_{n-2k+1},\hat{l}_{n-2k+2},\ldots,l_{n}),
\end{align*}
\noindent where $l_i$ are general linear forms in $S.$ Denote the ideals 
\[ J = \displaystyle{\bigcap_{i=1}^{2k}} L_i, \hspace{1cm} K = \displaystyle{\bigcap_{i=1}^{k}}  L_i, \hspace{1cm}  I = \displaystyle{\bigcap_{i=k+1}^{2k}}  L_i,\] 
so that $J = K \cap I.$ Let $R=S/J;$ to prove that $R$ is Koszul we will construct a Koszul filtration. To construct the filtration we need the two Hilbert series $H_{\left( J+(x_0) \right):(x_1)}(t)$ and $H_{\left( J+(l_0) \right):(l_1)}(t).$ We first calculate the former. Observe $(x_0,x_1) \subseteq L_i$ and $(l_0,l_1) \subseteq L_{k+i}$ for $i= 1,\ldots,k.$ Using the modular law {\cite[Chapter 1]{AM}}, we have the equality
\begin{equation}\label{eq:J+x_0=I}
\left( J+(x_0) \right):(x_1) = \left(K \cap I + K \cap (x_0) \right):(x_1) = (I+(x_0)):(x_1). 
\end{equation} \par
\noindent   So, it suffices to determine $H_{S/((I+(x_0)):(x_1))}(t)$. To this end, we first calculate $H_{S/(I+(x_0,x_1))}(t).$  To do so we use the short exact sequence
\begin{align*}
0 \rightarrow S/\left( I+(x_0) \right) \cap \left(I+(x_1) \right)  &\rightarrow S/\left( I+(x_0)\right) \oplus S/\left(I+(x_1)\right) \\ 
&\rightarrow S/\left(I+(x_0,x_1)\right) \rightarrow 0.
\end{align*}
Our assumption $m+1 \leq n \leq 2(m-1)$ guarantees that $\reg_S(S/I)=1.$ Thus, by Lemma $\ref{lem:hilbertSeries}$
\begin{equation}\label{eqn:S/I}
      H_{S/I}(t) = \frac{(1-k)t^2+2(k-1)t+1}{(1-t)^2},  
      \end{equation}  
 and since $x_0$ and $x_1$ are nonzerodivisors on $S/I,$  we have the following two Hilbert series
\[ H_{S/(I+(x_0))}(t) = H_{S/(I+(x_1))}(t) = \tfrac{(1-k)t^2+2(k-1)t+1}{1-t}.\]
Furthermore, the coordinate ring $S/(I+(x_0)) \cap (I+(x_1))$ corresponds precisely to a generic collection of $2k$ points. To see this, note that we are intersecting $k$ lines with two hyperplanes, where one hyperplane is defined by the ideal $(x_0)$ and the other is defined by the ideal $(x_1);$ none of the lines are contained in either hyperplane. Since these lines and hyperplanes are generic, the $2k$ points form a generic collection. We would like to note that we only need the $2k$ points to form a collection of points in general linear position, since by assumption $2k=m < n+1.$ So, by Theorem \ref{thm:hsForPoints}
\begin{align*}
H_{S/\left( ((I+(x_0)) \cap (I+(x_1)) \right)}(t) = \frac{(2k-1)t+1}{1-t}.
\end{align*}
By the additivity of the Hilbert series
\begin{align*}
    H_{S/(I+(x_0,x_1))}(t) &= H_{S/(I+(x_0))}(t) + H_{S/(I+(x_1))}(t) - H_{S/(I+(x_0)) \cap (I+(x_1))}(t) \\
    &= 2\left(\frac{(1-k)t^2+2(k-2)t+1}{1-t}\right) - \frac{(2k-1)t+1}{1-t} \\
    &= 1+2(k-1)t. 
\end{align*}
Thus, by the short exact sequence
\begin{align*}
0 \rightarrow  S/\left( (I+(x_0)):(x_1)\right)(-1) \rightarrow S/(I+(x_0)) \rightarrow S/(I+(x_0,x_1)) \rightarrow 0,
\end{align*}
Equation (\ref{eq:J+x_0=I}), and the additivity of the Hilbert series
\begin{align}\label{eq:colon}
H_{S/((J+(x_0)):(x_1))}(t)  &= H_{S/(I+(x_0)):(x_1)}(t) \\
\nonumber     &=\frac{1}{t} \left(  H_{S/(I+(x_0))}(t) - H_{S/(I+(x_0,x_1))} \right) \\
    \nonumber &= \frac{1}{t} \bigg( \frac{(1-k)t^2+2(k-1)t+1}{1-t} -1-2(k-1)t \bigg) \\
    \nonumber &= \frac{(k-1)t+1}{1-t}.
\end{align}
This gives us one our desired Hilbert series. An identical argument and interchanging $I$ with $K$ and $x_0$ and $x_1$ with $l_0$ and $l_1$ yields
\begin{equation}\label{eqn:S/K}
    H_{S/K}(t) = \frac{(1-k)t^2+2(k-1)t+1}{(1-t)^2}, 
    \end{equation}
    \begin{equation*}
       \hspace{1.9cm}  H_{S/(K+(l_0))}(t) = H_{S/(K+(l_1))}(t)  = \frac{(1-k)t^2+2(k-1)t+1}{(1-t)^2},
    \end{equation*}
and
\[H_{S/((J+(l_0)):(l_1))}(t)=H_{S/((J+(x_0)):(x_1))}(t).\]\par
\indent We can now define a Koszul filtration $\mathcal{F}$ for $R$.  We use $\bar{\cdot}$ to denote the image of an element of $S$ in $R=S/J$ for the remainder of the paper. We have already seen in Equation (\ref{eq:colon}) that 
\[ H_{S/(J+(x_0)):(x_1)}(t) = \frac{(k-1)t+1}{(1-t)} = 1+\sum_{i=0}^{\infty} k t^i.\]
Hence,  $n-k+1$ linearly independent linear forms are in a minimal generating set of $(J+(x_0)):(x_1).$ Clearly  $l_0,\ldots,l_{n-2k},x_0\in(J+(x_0)):(x_1),$ label $z_{n-2k+2},\ldots,z_{n-k}$ as the remaining linear forms from a minimal generating set of  $(J+(x_0)):(x_1).$ Similarly, choose $y_i$ from $(J+(l_0)):(l_1)$ so that $x_0,x_1,\ldots,x_{n-2k},l_0,y_{n-k+2},\ldots,y_{n-k}$ are linear forms forming a minimal generating set of $(J+(l_0)):(l_1).$ \par  
\indent The set $\{l_0,\ldots,l_{n-2k},x_{0},z_{n-2k+1},\ldots,z_{n-k},x_1\}$ is a linearly independent set over $S,$ otherwise $x_1^2 \in J+(x_0).$ This means $x_1^2 \in (L_i+(x_0))$ for $i = k+1,\ldots,2k,$ a contradiction. Similarly $\{x_0,\ldots,x_{n-2k},l_{0},y_{n-2k+1},\ldots,y_{n-k},l_1\}$ is linearly independent over $S.$ Let $w_{n-k+2},\ldots,w_{n+1},$ and $u_{n-k+2},\ldots,u_{n+1}$ be extensions of \[\{\bar{l}_0,\ldots,\bar{l}_{n-2k},\bar{x}_{0},\bar{z}_{n-2k+1},\ldots,\bar{z}_{n-k},\bar{x}_1\}\]
and
\[\{\bar{x}_0,\ldots,\bar{x}_{n-2k},\bar{l}_0,\bar{y}_{n-2k+1},\ldots,\bar{y}_{n-k},\bar{l}_1\}\] 
to minimal systems of generators of $\mathfrak{m}_R,$ respectively. Define $\mathcal{F}$ as follows
    \[ \mathcal{F} = \begin{cases} 
                     0, (\bar{x}_0), \hspace{.2cm } (\bar{x}_0,\bar{x}_1),  \hspace{3.3cm} (\bar{l}_0), \hspace{.2cm} (\bar{l}_0,\bar{l}_1),\\
                     \hspace{2cm}\vdots \hspace{5cm}\vdots
                     \\
                     (\bar{x}_0,\bar{x}_1,\ldots,\bar{x}_{n-2k}), \hspace{2.9cm} (\bar{l}_0,\bar{l}_1,\ldots,\bar{l}_{n-2k}),\\
                    (\bar{x}_0,\bar{x}_1,\ldots,\bar{x}_{n-2k},\bar{l}_0), \hspace{2.5cm} (\bar{l}_0,\bar{l}_1,\ldots,\bar{l}_{n-2k},\bar{x}_0),\\ 
                      (\bar{x}_0,\bar{x}_1,\ldots,\bar{x}_{n-2k},\bar{l}_0,\bar{y}_{n-2k+2}), \hspace{1cm} (\bar{l}_0,\bar{l}_1,\ldots,\bar{l}_{n-2k},\bar{x}_0,\bar{z}_{n-2k+2}),\\
                      
                     \hspace{4cm} \vdots \\
                       (\bar{x}_0,\bar{x}_1,\ldots,\bar{x}_{n-2k},\bar{l}_0,\bar{y}_{n-2k+2},\ldots,\bar{y}_{n-k}), \\
                    \hspace{1cm}   (\bar{l}_0,\bar{l}_1,\ldots,\bar{l}_{n-2k},\bar{x}_0,\bar{z}_{n-2k+2},\ldots,\bar{z}_{n-k}), \\

                  (\bar{x}_0,\bar{x}_1,\ldots,\bar{x}_{n-2k},\bar{l}_0,\bar{y}_{n-2k+2},\ldots,\bar{y}_{n-k},\bar{l}_{1}), \\ \hspace{1cm}  (\bar{l}_0,\bar{l}_1,\ldots,\bar{l}_{n-2k},\bar{x}_0,\bar{z}_{n-2k+2},\ldots,\bar{z}_{n-k},\bar{x}_{1}),\\
                  \hspace{4cm}\vdots\\
              (\bar{x}_0,\bar{x}_1,\ldots,\bar{x}_{n-2k},\bar{l}_0,\bar{y}_{n-2k+2},\ldots,\bar{y}_{n-k},\bar{l}_{1},u_{n-k+2}),\\
              \hspace{1cm}  (\bar{l}_0,\bar{l}_1,\ldots,\bar{l}_{n-2k},\bar{x}_0,\bar{z}_{n-2k+2}, \ldots,\bar{z}_{n-k},\bar{x}_{1},w_{n-k+2}),\\
                                \hspace{4cm}\vdots\\
              (\bar{x}_0,\bar{x}_1,\ldots,\bar{x}_{n-2k},\bar{l}_0,\bar{y}_{n-2k+2},\ldots,\bar{y}_{n-k},\bar{l}_{1},u_{n-k+2},\ldots,u_{n}),\\ 
              \hspace{1cm}  (\bar{l}_0,\bar{l}_1,\ldots,\bar{l}_{n-2k},\bar{x}_0,\bar{z}_{n-2k+2}, \ldots,\bar{z}_{n-k},\bar{x}_{1},w_{n-k+2},\ldots,w_{n}),\\
                     \mathfrak{m}_R
   \end{cases}
\]
\indent  We now prove $\mathcal{F}$ is a Koszul Filtration. We do this by proving several claims. Throughout the process we use the inclusion $(x_0,x_1) \cap (l_0,l_1) \subseteq J.$ Afterwards, we summarize all computed colons and list the claims that prove the calculated colons. \par 
\begin{clm}\label{clm:quadraticProp}
The ideal $(\bar{x}_0,\bar{x}_1,\bar{l}_0,\bar{l}_1)$ in $R$ has Hilbert series $H_{R/(\bar{x}_0,\bar{x}_1,\bar{l}_0,\bar{l}_1)}(t) = 1+(n-3)t$
and any ideal $P$ containing this ideal has the property that $P:(\ell) = \mathfrak{m}_R,$ where $\ell$ is a linear form not contained in $P.$
\end{clm}
\begin{proof}
\indent We begin by observing that our assumption $m+1 \leq n \leq 2(m-1)$ and Proposition \ref{thm:reg} yield $\mathrm{reg}_S(S/J)=2.$ Thus, by Lemma $\ref{lem:hilbertSeries}$
\[    H_{S/J}(t) = \frac{(n+1-4k)t^3+(6k-2n-1)t^2+(n-1)t+1}{(1-t)^2}.\]
Now, $L_i:(x_0) = L_i$ for $i=k+1,\ldots,2k,$ since $x_0 \notin L_i.$ Thus, 
\[J:(x_0) = \left( \bigcap_{i=1}^{2k} L_i \right) :(x_0) =\bigcap_{i=1}^{2k} \left(  L_i :(x_0)  \right ) = \bigcap_{i=k+1}^{2k}  L_i = I.\]
So, $H_{S/(J:(x_0))}(t)=H_{S/I}(t).$  Using the short exact sequence
\[
0 \rightarrow S/(J:(x_0)) (-1) \rightarrow S/J \rightarrow S/(J+(x_0)) \rightarrow 0, \\
\]
Equation (\ref{eqn:S/I}), and the additivity of the Hilbert series yields 
\begin{align*}
H_{S/(J+(x_0))}(t) &= H_{S/J}(t)-tH_{S/(J:(x_0)) }(t)  \\
&=H_{S/J}(t)-tH_{S/I }(t)  \\
&=\tfrac{(1+n-4k)t^3+(6k-2n-1)t^2+(n-1)t+1}{(1-t)^2} -t\left( \tfrac{(1-k)t^2+2(k-1)t+1}{(1-t)^2}\right)\\
                             &= \frac{(n-3k)t^3+(4k-2n+1)t^2+(n-2)t+1}{(1-t)^2}.
\end{align*}\par
\noindent Using the short exact sequence
\begin{align*}
    0 \rightarrow  S/((J+(x_0)):(x_1))(-1) \rightarrow S/(J+(x_0)) \rightarrow S/(J+(x_0,x_1)) \rightarrow 0,
\end{align*}
Equation (\ref{eq:colon}), the previous Hilbert series, and the additivity of the Hilbert series yields
\begin{align*}
    H_{S/(J+(x_0,x_1))}(t) &= H_{S/(J+(x_0))}(t) - tH_{S/(J+(x_0)):(x_1)}(t) \\
    &= \tfrac{(n-3k)t^3+(4k-2n+1)t^2+(n-2)t+1}{(1-t)^2} - t\left( \tfrac{(k-1)t+1}{1-t} \right) \\
    &= \frac{(n-2k-1)t^3+(3k-2n+3)t^2+(n-3)t+1}{(1-t)^2}.
\end{align*}
Replacing $x_0$ and $x_1$ with $l_0$ and $l_1$ demonstrates that
\[ H_{S/(J+(x_0,x_1))}(t) =  H_{S/(J+(l_0,l_1))}(t).\]
Thus, using the short exact sequence
\begin{align*}
0 \rightarrow  R/((\bar{x}_0,\bar{x}_1) \cap (\bar{l}_0,\bar{l}_1)) &\rightarrow R/(\bar{x}_0,\bar{x}_1) \oplus R/(\bar{l}_0,\bar{l}_1) \\ &\rightarrow R/(\bar{x}_0,\bar{x}_1,\bar{l}_0,\bar{l}_1) \rightarrow 0,
\end{align*}
and the additivity of the Hilbert series yields
\begin{align*}
H_{R/(\bar{x}_0,\bar{x}_1,\bar{l}_0,\bar{l}_1)}(t) &= H_{R/(\bar{x_0},\bar{x_1})}(t)+H_{R/(\bar{l_0},\bar{l_1})}(t)-H_{R}(t) \\
                        \nonumber            &= 2\left(\tfrac{(n-2k-1)t^3+(3k-2n+3)t^2+(n-3)t+1}{(1-t)^2}\right)  \\
                        \nonumber            &\hspace{3cm}-\tfrac{(1+n-4k)t^3+(6k-2n-1)t^2+(n-1)t+1}{(1-t)^2} \\
                        \nonumber    &= \frac{(n-3)t^3+(7-2n)t^2+(n-5)t+1}{(1-t)^2} \\
                        \nonumber     &= 1+(n-3)t. 
\end{align*}
So, $R_2 \subset (\bar{x}_0,\bar{x}_1,\bar{l}_0,\bar{l}_1).$ This means that any ideal $P \subset R$ containing the ideal $(\bar{x}_0,\bar{x}_1,\bar{l}_0,\bar{l}_1)$ has the property that $P:(\ell) = \mathfrak{m}_R,$ where $\ell$ is a linear form not contained in $P$. 
\end{proof}
\begin{clm}\label{clm:J+x's=K}
  For $i = 1,\ldots,n-2k,$ we have the two Hilbert series
  \begin{align*}
     H_{R/(\bar{x}_0,\bar{x}_1,\bar{x}_2,\ldots,\bar{x}_i)}(t) &= H_{R/(\bar{l}_0,\bar{l}_1,\bar{l}_2,\ldots,\bar{l}_i)}(t) \\
  &=\tfrac{(n-2k-i)t^3+(3k-2n+2i+1)t^2+(n-(i+2))t+1}{(1-t)^2},
    \end{align*}
   and the two equalities $J+(x_0,\ldots,x_{n-2k}) = K,$ and $J+(l_0,\ldots,l_{n-2k}) = I.$ 
\end{clm}
\begin{proof}
\indent Adding the linear forms $x_2,\ldots,x_i$ to the ideal $(\bar{x}_0,\bar{x}_1,\bar{l}_0,\bar{l}_1)$ yields
$$H_{R/(\bar{x}_0,\bar{x}_1,\bar{l}_0,\bar{l}_1,\bar{x}_2,\ldots,\bar{x}_{i})}(t) = H_{R/(\bar{x}_0,\bar{x}_1,\bar{l}_0,\bar{l}_1,\bar{l}_2,\ldots,\bar{l}_{i})}(t) = 1+(n-(i+2))t$$
for $i =  2,\ldots,n-2k$. Using the short exact sequence
  \begin{align*}
  0 \rightarrow  R /( (\bar{x}_0,\bar{x}_1,\bar{x}_2) \cap (\bar{l}_0,\bar{l}_1)) &\rightarrow R/(\bar{x}_0,\bar{x}_1,\bar{x}_2)  \oplus R/(\bar{l}_0,\bar{l}_1)\\
  &\rightarrow R/(\bar{x}_0,\bar{x}_1,\bar{x}_2,\bar{l}_0,\bar{l}_1) \rightarrow 0
   \end{align*}
 and the additivity of the Hilbert series gives
\begin{align*}
H_{R/(\bar{x}_0,\bar{x}_1,\bar{x}_2)}(t) &= H_{R/(\bar{x}_0,\bar{x}_1,\bar{l}_0,\bar{l}_1,\bar{x}_2) }(t) + H_{R}(t)- H_{R/(\bar{l}_0,\bar{l}_1)}(t)\\
                                          &= 1+(n-4)t +\tfrac{(1+n-4k)t^3+(6k-2n-1)t^2+(n-1)t+1}{(1-t)^2}  \\
                                          &\hspace{3.5cm} -\tfrac{(n-2k-1)t^3+(3k-2n+3)t^2+(n-3)t+1}{(1-t)^2}\\
                                          &=1+(n-4)t+\frac{(2-2k)t^3+(3k-4)t^2+2t}{(1-t)^2} \\
                                          &= \frac{(n-2k-2)t^3+(3k-2n+5)t^2+(n-4)t+1}{(1-t)^2}.
\end{align*}
Replacing $(\bar{x}_0,\bar{x}_1,\bar{x}_2)$ with $(\bar{x}_0,\bar{x}_1,\bar{x}_2,\bar{x}_3)$ in the above short exact sequence and using the additivity of the Hilbert series yields
\begin{align*}
H_{R/(\bar{x}_0,\bar{x}_1,\bar{x}_2,\bar{x}_3)}(t) &= H_{R/(\bar{x}_0,\bar{x}_1,\bar{l}_0,\bar{l}_1,\bar{x}_2,\bar{x}_3) }(t) + H_{R}(t)- H_{R/(\bar{l}_0,\bar{l}_1)}(t)\\
                                            &= 1+(n-5)t +\tfrac{(1+n-4k)t^3+(6k-2n-1)t^2+(n-1)t+1}{(1-t)^2}  \\
                                          &\hspace{3.5cm} -\tfrac{(n-2k-1)t^3+(3k-2n+3)t^2+(n-3)t+1}{(1-t)^2}
                                                                                    \end{align*}    
                                                                                    \begin{align*}
                                          &= 1+(n-5)t + \frac{(2-2k)t^3+(3k-4)t^2+2t}{(1-t)^2} \\
                                          &= \frac{(n-2k-3)t^3+(3k-2n+7)t^2+(n-5)t+1}{(1-t)^2} 
                                          \end{align*}      
By induction
\begin{equation}\label{eqn:hilbertSeriesXs}
    H_{R/(\bar{x}_0,\bar{x}_1,\bar{x}_2,\ldots,\bar{x}_i)}(t) = \tfrac{(n-2k-i)t^3+(3k-2n+2i+1)t^2+(n-(i+2))t+1}{(1-t)^2}
\end{equation}
for $i =2,\ldots,n-2k.$ Setting $i=n-2k$ we obtain the Hilbert series
                                          \begin{align*}
H_{R/(\bar{x}_0,\ldots,\bar{x}_{n-2k})}(t) &= H_{R/(\bar{x}_0,\bar{x}_1,\bar{l}_0,\bar{l}_1,\bar{x}_2,\ldots,\bar{x}_{n-2k}) }(t)  \\
    \nonumber                                                 & \hspace{3cm}+H_{R}(t)- H_{R/(\bar{l}_0,\bar{l}_1)}(t)\\
    \nonumber                                        &= 1+(2k-2)t +\tfrac{(1+n-4k)t^3+(6k-2n-1)t^2+(n-1)t+1}{(1-t)^2}  \\
    \nonumber                                         &\hspace{2.9cm} -\tfrac{(n-2k-1)t^3+(3k-2n+3)t^2+(n-3)t+1}{(1-t)^2}\\
    \nonumber                                         &= 1+(2k-2)t + \frac{(2-2k)t^3+(3k-4)t^2+2t}{(1-t)^2} \\
    \nonumber                                         &= \frac{(1-k)t^2+2(k-1)t+1}{(1-t)^2}. 
\end{align*}
Interchanging each $x_i$ with $l_i$ gives us the other desired Hilbert series. \par
\indent The Hilbert series in Equation (\ref{eqn:hilbertSeriesXs}) is the same as in Equation (\ref{eqn:S/K}). Furthermore $J+(x_0,\ldots,x_{n-2k})\subseteq K.$ So, we have that $J+(x_0,\ldots,x_{n-2k})= K$ and interchanging each $x_i$ with $l_i$ gives us the other equality.
\end{proof}
\begin{clm}\label{eqn:claim1} We have the equalities
\[(\bar{l}_0,\ldots,\bar{l}_{n-2k},\bar{x}_0):(\bar{z}_{n-2k+2}) = \mathfrak{m}_R,\]
\[(\bar{x}_0,\ldots,\bar{x}_{n-2k},\bar{l}_0):(\bar{y}_{n-2k+2}) = \mathfrak{m}_R,\]
and
\[(\bar{l}_0,\ldots,\bar{l}_{n-2k},\bar{x}_0,\bar{z}_{n-2k+2},\ldots,\bar{z}_{i}):(\bar{z}_{i+1}) = \mathfrak{m}_R,\]
\[(\bar{x}_0,\ldots,\bar{x}_{n-2k},\bar{l}_0,\bar{y}_{n-2k+2},\ldots,\bar{y}_{i}):(\bar{y}_{i+1}) = \mathfrak{m}_R,\]
for $i=n-2k+2,\ldots,n-k.$ Furthermore, 
\[(\bar{x}_0):(\bar{x}_1) = (\bar{l}_0,\ldots,\bar{l}_{n-2k},\bar{x}_0,\bar{z}_{n-2k+2},\ldots,\bar{z}_{n-k})\]
and 
\[(\bar{l}_0):(\bar{l}_1)=(\bar{x}_0,\ldots,\bar{x}_{n-2k},\bar{l}_0,\bar{y}_{n-2k+2},\ldots,\bar{y}_{n-k}).\]

\end{clm}

\begin{proof}
We begin by observing 
\[(\bar{l}_0,\bar{l}_{1},\bar{x}_0,\bar{x}_1) \subseteq (\bar{l}_0,\bar{l}_{1},\bar{x}_0,\bar{l}_2,\ldots,\bar{l}_{n-2k}):(\bar{z}_{n-2k+2}).\] So by Claim \ref{clm:quadraticProp}, we conclude that
$$H_{R/((\bar{l}_0,\bar{l}_{1},\bar{l}_2,\cdots,\bar{l}_{n-2k},\bar{x}_0):(\bar{z}_{n-2k+2}))}(t) = 1+\alpha t,$$
where $\alpha \in \{0,1,\ldots,n-3\}.$ Using the short exact sequence
\begin{align} \label{ses:colon}
\vspace{-1cm} 0 \rightarrow  R/((\bar{x}_0,\bar{l}_0,\ldots,\bar{l}_{n-2k}):(\bar{z}_{n-2k+2}))(-1) &\rightarrow R/(\bar{x}_0,\bar{l}_0,\ldots,\bar{l}_{n-2k}) \\
\nonumber& \hspace{-3cm} \rightarrow R/(\bar{x}_0,\bar{l}_0,\ldots,\bar{l}_{n-2k},\bar{z}_{n-2k+2}) \rightarrow 0,
\end{align}
Claim \ref{clm:J+x's=K}, and that the fact that $x_0$ is a nonzerodivisor on $S/I$  we obtain
\begin{align*}
    H_{R/(\bar{x}_0,\bar{l}_0,\ldots,\bar{l}_{n-2k},\bar{z}_{n-2k+2})}(t) &=  H_{R/(\bar{x}_0,\bar{l}_0,\ldots,\bar{l}_{n-2k})}(t) \\
    &\hspace{1cm}-  tH_{R/(\bar{x}_0,\bar{l}_0,\ldots,\bar{l}_{n-2k}):(\bar{z}_{n-2k+2})}(t) \\
    &= \frac{(1-k)t^2+2(k-1)t+1}{(1-t)} - t(1+\alpha t) \\\
    &= \frac{\alpha t^3+(2-\alpha-k)t^2+(2k-3)t+1}{(1-t)} \\
    &= 1+(2k-2)t+(k-\alpha)t^2 + \sum_{j=3}^{\infty} k t^j.
\end{align*}
 We also have the containment
\begin{align*}
    (\bar{l}_0,\ldots,\bar{l}_{n-2k},\bar{x}_0,\bar{z}_{n-2k+2}) \subseteq (\bar{l}_0,\ldots,\bar{l}_{n-2k},\bar{x}_0):(\bar{x}_1).
    \end{align*}
    Using Claim \ref{clm:J+x's=K} we obtain the equality
$$ (J + (l_0,\ldots,l_{n-2k},x_0)):(x_1) = (I+(x_0)):(x_1),$$
which has Hilbert series computed in $(\ref{eq:colon}).$ Hence
\begin{align*}
    H_{R/((\bar{l}_0,\bar{l}_{1},\bar{l}_2,\ldots,\bar{l}_{n-2k},\bar{x}_0):(\bar{x}_{1}))}(t)  &= \frac{(k-1)t+1}{1-t} = 1 +  \sum_{j=1}^{\infty} k t^j.
\end{align*}
Comparing coefficients of $t^2$ yields $k \leq k - \alpha,$ and so $\alpha=0,$ proving the first equality. We immediately have the equality
\begin{equation}
    (\bar{l}_0,\ldots,\bar{l}_{n-2k},\bar{x}_0,\bar{z}_{n-2k+2},\ldots,\bar{z}_{i}):(\bar{z}_{i+1}) = \mathfrak{m}_R, 
\end{equation}
for each $i =n-2k+2,\ldots, n-k.$ \par 
Notice that setting  $\alpha = 0,$ yields
\begin{align*}
H_{R/(\bar{l}_0,\ldots,\bar{l}_{n-2k},\bar{x}_0,\bar{z}_{n-2k+2})}(t) = \frac{(2-k)t^2+(2k-3)t+1}{(1-t)}. \end{align*}
Denote $V_i$ and $V_i'$ to be the ideals
\begin{align*}
V_i &= (\bar{l}_0,\ldots,\bar{l}_{n-2k},\bar{x}_0,\bar{z}_{n-2k+2},\ldots,\bar{z}_{i})
\end{align*}
and
\begin{align*}
V_i' &= (\bar{x}_0,\ldots,\bar{x}_{n-2k},\bar{l}_0,\bar{y}_{n-2k+2},\ldots,\bar{y}_{i})
\end{align*}
for $i = n-2k+2,\ldots,n-k$.
Replacing the ideals in (\ref{ses:colon}) with the three ideals $V_{n-2k+3},V_{n-2k+2},$ and $V_{n-2k+2}:(\bar{z}_{n-2k+3})$, and using the additivity of the Hilbert series yields
    \begin{align*}
    H_{R/V_{n-2k+3}}(t) &= H_{R/V_{n-2k+2}}(t) -t H_{R/V_{n-2k+2}:(\bar{z}_{n-2k+3})}(t) \\
    &= \frac{(2-k)t^2+(2k-3)t+1}{(1-t)} - t \\
    &= \frac{(3-k)t^2+(2k-4)t+1}{(1-t)}. 
            \end{align*}
Continuing in this fashion gives
    \begin{align*}
     H_{R/V_{n-k}}(t) &= \frac{(k-1)t+1}{(1-t)}. 
\end{align*}
So, both $(\bar{x}_0):(\bar{x}_1)$ and $V_{n-k}$ have the same Hilbert series. Furthermore, $V_{n-k}\subseteq (\bar{x}_0):(\bar{x}_1).$
So these ideals are in fact equal. Interchanging $x_0$ and $x_1$ with $l_0$ and $l_1$ yields the remaining equality. \par  
\end{proof}
\begin{clm}\label{eqn:claim2} We have the equalities
 \[(0_R):(\bar{x}_0) = (\bar{l}_0,\ldots,\bar{l}_{n-2k})\]
 and
  \[(0_R):(\bar{l}_0) = (\bar{x}_0,\ldots,\bar{x}_{n-2k}).\]
\end{clm}
\begin{proof}
\indent The two equalities follow immediately since $(0_R):(\bar{x}_0) \subseteq (\bar{l}_0,\ldots,\bar{l}_{n-2k})$ and $(0_R):(\bar{l}_0) \subseteq (\bar{x}_0,\ldots,\bar{x}_{n-2k})$ and all four ideals have the same Hilbert series by Claim \ref{clm:J+x's=K}.
\end{proof}
\begin{clm}\label{eqn:claim5} We have the equality
\[(\bar{x}_0,\bar{x}_1,\bar{x}_2,\ldots,\bar{x}_i):(\bar{x}_{i+1})=\mathfrak{m}_R,\]
for $i=2,\ldots,n-2k-1.$
\end{clm}
\begin{proof}
Using the short exact sequence
\begin{align*} 
 0 \rightarrow  R/\left((\bar{x}_0,\ldots,\bar{x}_{i}):(\bar{x}_{i+1})\right)(-1) \rightarrow R/(\bar{x}_0,\ldots,\bar{x}_{i}) \\
\nonumber& \hspace{-3cm} \rightarrow R/(\bar{x}_0,\ldots,\bar{x}_{i},\bar{x}_{i+1}) \rightarrow 0,
\end{align*}
and the Hilbert series from (\ref{eqn:hilbertSeriesXs}), we get
\begin{align*}
    H_{R/((\bar{x}_0,\ldots,\bar{x}_i):(\bar{x}_{i+1}))}(t) &= \frac{1}{t} \Bigg( H_{R/(\bar{x}_0,\ldots,\bar{x}_i)}(t) - H_{R/(\bar{x}_0,\ldots,\bar{x}_i,\bar{x}_{i+1})}(t) \Bigg) \\
    &= \frac{1}{t} \bigg( \tfrac{(n-2k-i)t^3+(3k-2n+2i+1)t^2+(n-(i+2))t+1}{(1-t)^2}   \\
    & \hspace{1cm} -\tfrac{(n-2k-i-1)t^3+(3k-2n+2i+3)t^2+(n-(i+3))t+1}{(1-t)^2}\bigg) \\
    &= \frac{t^2-2t+1}{(1-t)^2} \\
    &= 1,
\end{align*}
proving the claim.
\end{proof}
\begin{clm} \label{eqn:claim3} We have the four equalities
\[(\bar{x}_0,\ldots,\bar{x}_{n-2k}):(\bar{l}_0)=(\bar{x}_0,\ldots,\bar{x}_{n-2k}),\]
\[(\bar{l}_0,\ldots,\bar{l}_{n-2k}):(\bar{x}_0)=(\bar{l}_0,\ldots,\bar{l}_{n-2k}),\]
\[V_{n-k}':(\bar{l}_1) = V_{n-k}',\]
\[V_{n-k}:(\bar{x}_1) = V_{n-k}.\]
\end{clm}
\begin{proof}

\indent  The equality
\[(\bar{x}_0,\ldots,\bar{x}_{n-2k}):(\bar{l}_0)=(\bar{x}_0,\ldots,\bar{x}_{n-2k})\]
follows from the genericity of $l_0.$ We now aim to show the equality 
\[V_{n-k}:(\bar{x}_1) = V_{n-k}.\]
We always have the containment $V_{n-k} \subseteq V_{n-k}:(\bar{x}_1)$
and by Claim \ref{eqn:claim1} we have already determined $H_{R/V_{n-k}}(t).$ So, we must only determine 
$H_{R/(V_{n-k}:(\bar{x}_1))}(t).$ We aim to use the additivity of the Hilbert series along the short exact sequence
\begin{align}\label{ses:sesForV}
0 \rightarrow  R/(V_{n-k}:(\bar{x}_1))(-1) \rightarrow R/V_{n-k} \rightarrow R/(V_{n-k}+(\bar{x}_1)) \rightarrow 0,
\end{align}
but we first must determine $H_{R/(V_{n-k}+(\bar{x}_1))}(t).$ By Claim \ref{clm:quadraticProp}, adding the linear forms $\bar{l}_2,\ldots,\bar{l}_{n-2k},\bar{z}_{n-2k+2},\ldots,\bar{z}_{n-k}$ to the ideal $(\bar{x}_0,\bar{x}_1,\bar{l}_0,\bar{l}_1)$
yields the Hilbert series
\begin{align*}
    H_{R/(V_{n-k}+(\bar{x}_1)))}(t) &= 1+(k-1)t.
\end{align*}
Using Claim \ref{eqn:claim1} and the additivity of the Hilbert series along the short exact sequence (\ref{ses:sesForV}) yields
\begin{align*}
H_{R/(V_{n-k}:(\bar{x}_1))}(t) &= \frac{1}{t} \Bigg( H_{R/V_{n-k}}(t) -H_{R/(V_{n-k}+(\bar{x}_1))}(t) \Bigg) \\
&= \frac{1}{t}\bigg(\frac{(k-1)t+1}{(1-t)}- (1+(k-1)t)\bigg)\\
&= \frac{(k-1)t+1}{(1-t)}.
\end{align*}
Interchanging $\bar{x}_0$ and  $\bar{x}_1$ with  $\bar{l}_0$ and  $\bar{l}_1,$ proves the other two equalities.
\end{proof}
Below is a list of calculated colons with the corresponding justification.\\
    $\begin{cases} 
                     (0):(\bar{x}_0) =(\bar{l}_0,\dots,\bar{l}_{n-2k}),  \hspace{1cm} (0):(\bar{l}_0)=(\bar{x}_0,\ldots,\bar{x}_{n-2k}), \hspace{.2cm} \ref{eqn:claim2}  \\
                     
                     (\bar{x}_0):(\bar{x}_1) =  (\bar{l}_0,\ldots,\bar{l}_{n-2k},\bar{x}_0,\bar{z}_{n-2k+2},\ldots,\bar{z}_{n-k}),   \hspace{.2cm} \ref{eqn:claim1} \\ \hspace{.3cm} (\bar{l}_0):(\bar{l}_1)=(\bar{x}_0,\ldots,\bar{x}_{n-2k},\bar{l}_0,\bar{y}_{n-2k+2},\ldots,\bar{y}_{n-k}),  \hspace{.2cm} \ref{eqn:claim1} \\
                     (\bar{x}_0,\bar{x}_1,\bar{x}_2,\ldots,\bar{x}_i):(\bar{x}_{i+1})=\mathfrak{m}_R, \hspace{.5cm} i=2,\ldots, n-2k-1, \hspace{.2cm}\ref{eqn:claim5} \\ 
                       \hspace{.5cm} (\bar{l}_0,\bar{l}_1,\bar{l}_2,\ldots,\bar{l}_i):(\bar{l}_{i+1})=\mathfrak{m}_R, \hspace{.5cm}   i=2,\ldots, n-2k-1, \hspace{.2cm}\ref{eqn:claim5}  \\
                     (\bar{x}_0,\bar{x}_1,\bar{x}_2,\ldots,\bar{x}_{n-2k}):(\bar{l}_{0})=(\bar{l}_0,\ldots,\bar{l}_{n-2k}),  \hspace{.2cm}  \ref{eqn:claim3}\\
                     \hspace{.5cm} (\bar{l}_0,\bar{l}_1,\bar{l}_2,\ldots,\bar{l}_{n-2k}):(\bar{x}_{0})=(\bar{x}_0,\ldots,\bar{x}_{n-2k}),  \hspace{.2cm} \ref{eqn:claim3} \\
                      (\bar{x}_0,\ldots,\bar{x}_{n-2k},\bar{l}_0):(\bar{y}_{n-2k+2}) =  \mathfrak{m}_R,  \hspace{.2cm}  \ref{eqn:claim1} \\
                         \hspace{.5cm} (\bar{l}_0,\ldots,\bar{l}_{n-2k},\bar{x}_0):(\bar{z}_{n-2k+2}) =  \mathfrak{m}_R, \hspace{.2cm}  \ref{eqn:claim1} \\
                     V_{i}':(\bar{y}_{i+1})= \mathfrak{m}_R, \hspace{.3cm}  V_{i}:(\bar{z}_{i+1})= \mathfrak{m}_R,  \hspace{.1cm} i=n-2k+2,\ldots, n-k-1, \hspace{.2cm} \ref{eqn:claim1} \\
                     
                    \hspace{.3cm}  V_{n-k}':(\bar{l}_1)=V_{n-k}',\hspace{.3cm}  V_{n-k}:(\bar{x}_1)= V_{n-k}, \hspace{.2cm}  \ref{eqn:claim3} \\

                     (V_{n-k}'+(\bar{l}_1)):(\bar{u}_{n-k+1})=\mathfrak{m}_R, \hspace{.3cm}    \hspace{.2cm}(V_{n-k}+(\bar{x}_1)):(\bar{w}_{n-k+1})=\mathfrak{m}_R, \hspace{.2cm} \ref{clm:quadraticProp}  \\

                      \hspace{.3cm}(V_{n-k}'+(\bar{l}_1,\bar{u}_{n-k+1},\ldots,\bar{u}_i)):(\bar{u}_{i+1})=\mathfrak{m}_R, \hspace{.01cm} i = n-k+1,\ldots,n-1, \hspace{.2cm} \ref{clm:quadraticProp}  \\
                      (V_{n-k}+(\bar{x}_1,\bar{w}_{n-k+1},\ldots,\bar{w}_i)):(\bar{w}_{i+1})=\mathfrak{m}_R, \hspace{.01cm}  i = n-k+1,\ldots,n-1, \hspace{.2cm}\ref{clm:quadraticProp}.
   \end{cases}$\\
   
   This completes the proof of Theorem \ref{thm:koszulThm}.
\end{proof}

There is at least one example of a coordinate ring with the Koszul property which is not covered by our previous theorem. Let $\mathcal{M}$ be a generic collection of $5$ lines in $\mathbb{P}^6.$ By Remark \ref{rmk:linesIntersection} and a change of basis we may assume the defining ideals for our $5$ lines have the following form
\begin{align*}
L_1 &= ( x_0,x_3,x_4,x_5,x_6) \hspace{3.1cm} L_2 = (x_0,x_1,x_4,x_5,x_2+ax_3+x_6) \\
L_3 &= (x_0,x_1,x_2,x_6,x_3+bx_4+x_5) \hspace{1.25cm} L_4 = (x_1,x_2,x_3,x_5,x_0+x_4+x_6) \\
& \hspace{3cm}L_5 = (x_2,x_3,x_4,x_6,x_0+x_1+x_5),
\end{align*}
where $a,b \in \mathbb{C}$ are algebraically independent over $\mathbb{Q}.$ Some further explanation is needed why we may assume our $5$ lines have this form. \par 
\indent By Remark \ref{rmk:linesIntersection}, the intersection of any triple of the defining ideals of our $5$ lines contains a single linear form in a minimal generating set. Furthermore, the intersection of any pair of defining ideals for our $5$ contains $3$ linear forms in a minimal generating set. Thus, after a change of basis we may assume 
\begin{align*}
L_1 &= ( x_0,x_3,x_4,x_5,x_6) \hspace{1.5cm} L_2 = (x_0,x_1,x_4,x_5,l_0) \\
 L_3 &= (x_0,x_1,x_2,x_6,l_1) \hspace{1.6cm} L_4 = (x_1,x_2,x_3,x_5,l_2) \\
& \hspace{2.7cm}L_5 = (x_2,x_3,x_4,x_6,l_3).
\end{align*}
where the linear forms $l_0,\ldots,l_3$ have the form
\begin{align*}
    l_0 &= c_{0,2}x_2+c_{0,3}x_3+c_{0,6}x_6, \hspace{1cm} l_1 = c_{1,3}x_3+c_{1,4}x_4+c_{1,5}x_5, \\
    l_2 &= c_{2,0}x_0+c_{2,4}x_4+c_{2,6}x_6, \hspace{1cm} l_3 = c_{3,0}x_0+c_{3,1}x_1+c_{3,5}x_5.
\end{align*}
It is of no loss to assume these are all monic in certain indeterminates. That is they have the form
\begin{align*}
    l_0 &= c_{0,2}x_2+c_{0,3}x_3+x_6, \hspace{1cm} l_1 = x_3+c_{1,4}x_4+c_{1,5}x_5, \\
    l_2 &= x_0+c_{2,4}x_4+c_{2,6}x_6, \hspace{1cm} l_3 = c_{3,0}x_0+c_{3,1}x_1+x_5.
\end{align*}
Through a change of basis we may reduce the coefficient on $x_5$ in $l_1$ to $1,$ and then normalize $l_3$ to be monic in $x_5;$ then through another change of basis we may reduce the coefficient on $x_0$ in $l_3$  to $1,$ and then normalize $l_2$ to be monic in $x_0;$ then through another change of basis we may reduce the coefficient on $x_6$ in $l_2$ to $1,$ and then normalize $l_0$ to be monic in $x_6;$ then through another change of basis we may reduce the coefficient on $x_2$ in $l_0$  to $1;$  then through another change of basis we may reduce the coefficient on $x_4$ in $l_2$  to $1;$ then through another change of basis we may reduce the coefficient on $x_1$ in $l_3$ to $1.$ Ultimately, we obtain
\begin{align*}
    l_0 &= x_2+c_{0,3}x_3+x_6, \hspace{1cm} l_1 = x_3+c_{1,4}x_4+x_5, \\
    l_2 &= x_0+x_4+x_6, \hspace{1.6cm} l_3 = x_0+x_1+x_5.
\end{align*}
Note the order in which we make these reductions is important.
\begin{prop}\label{prop:exampleKoszul}
Let $\mathcal{M}$ be a generic collection of $5$ lines in $\mathbb{P}^6$ and $R$ the coordinate ring. Then $R$ is Koszul.
\end{prop}
\begin{proof} After a change of basis we may represent the defining ideal for our $5$ lines as above. Below is a Koszul filtration   \par

\[\mathcal{F}= \begin{cases} 
 (0_R), (\bar{x}_0), (\bar{x}_2),
 (\bar{x}_0,\bar{x}_4),(\bar{x}_0,\bar{x}_1),(\bar{x}_2,\bar{x}_3),(\bar{x}_0,\bar{x}_6),(\bar{x}_0,\bar{x}_2,\bar{x}_3)\\
 (\bar{x}_2,\bar{x}_3,\bar{x}_0+\bar{x}_1+\bar{x}_4+\bar{x}_5+\bar{x}_6),(\bar{x}_0,\bar{x}_4,\bar{x}_5),(\bar{x}_0,\bar{x}_1,\bar{x}_2),\\(\bar{x}_0,\bar{x}_3,\bar{x}_4),
 (\bar{x}_0,\bar{x}_1,\bar{x}_4),(\bar{x}_0,\bar{x}_2,\bar{x}_4),(\bar{x}_0,\bar{x}_2,\bar{x}_6),(\bar{x}_0,\bar{x}_1,\bar{x}_5),\\
 (\bar{x}_0,\bar{x}_1,\bar{x}_2,\bar{x}_3+b\bar{x}_4+\bar{x}_5+b\bar{x}_6),(\bar{x}_0,\bar{x}_1,\bar{x}_2,\bar{x}_3+b\bar{x}_4+\bar{x}_5+\frac{1}{a}\bar{x}_6),\\
 (\bar{x}_0,\bar{x}_3,\bar{x}_4,\bar{x}_6),(\bar{x}_0,\bar{x}_1,\bar{x}_4,\bar{x}_5), (\bar{x}_0,\bar{x}_1,\bar{x}_4,\bar{x}_2+a\bar{x}_3+a\bar{x}_5+\bar{x}_6), \\
 (\bar{x}_0,\bar{x}_2,\bar{x}_4,\bar{x}_6),(\bar{x}_0,\bar{x}_1,\bar{x}_2,\bar{x}_6),
 (\bar{x}_0,\bar{x}_2,\bar{x}_3,\bar{x}_6),(\bar{x}_0,\bar{x}_1,\bar{x}_5,\bar{x}_6),\\
 (\bar{x}_0,\bar{x}_1,\bar{x}_2,\bar{x}_5), (\bar{x}_0,\bar{x}_4,\bar{x}_5,\bar{x}_6),
 (\bar{x}_0,\bar{x}_1,\bar{x}_5,\bar{x}_2+a\bar{x}_3+\bar{x}_4+\bar{x}_6),\\
 (\bar{x}_0,\bar{x}_1,\bar{x}_2,\bar{x}_5,\bar{x}_3+b\bar{x}_4+b\bar{x}_6),
 (\bar{x}_0,\bar{x}_2,\bar{x}_4,\bar{x}_6,\bar{x}_1+\bar{x}_3+\bar{x}_5),\\
 (\bar{x}_0,\bar{x}_2,\bar{x}_3,\bar{x}_6,\bar{x}_4+\frac{1}{b}\bar{x}_5),(\bar{x}_0,\bar{x}_3,\bar{x}_4,\bar{x}_5,\bar{x}_6),\\
 (\bar{x}_0,\bar{x}_1,\bar{x}_2,\bar{x}_6,\bar{x}_3+b\bar{x}_4+\bar{x}_5),(\bar{x}_0,\bar{x}_1,\bar{x}_2,\bar{x}_4,\bar{x}_5),\\ (\bar{x}_0,\bar{x}_2,\bar{x}_3,\bar{x}_4,\bar{x}_6),
 (\bar{x}_0,\bar{x}_1,\bar{x}_2,\bar{x}_3,\bar{x}_5),\\
 (\bar{x}_0,\bar{x}_1,\bar{x}_2,\bar{x}_5,\bar{x}_3+\frac{1}{a}\bar{x}_4+\frac{1}{a}\bar{x}_6),(\bar{x}_0,\bar{x}_1,\bar{x}_2,\bar{x}_5,\bar{x}_3+b\bar{x}_4+\frac{1}{a}\bar{x}_6),\\
 (\bar{x}_0,\bar{x}_1,\bar{x}_3,\bar{x}_4,\bar{x}_5),(\bar{x}_0,\bar{x}_1,\bar{x}_2,\bar{x}_6,\bar{x}_3+b\bar{x}_4),\\
 (\bar{x}_0,\bar{x}_1,\bar{x}_4,\bar{x}_5,\bar{x}_2+a\bar{x}_3+\bar{x}_6),(\bar{x}_0,\bar{x}_1,\bar{x}_5,\bar{x}_6,\bar{x}_3+b\bar{x}_4),\\
 (\bar{x}_0,\bar{x}_1,\bar{x}_4,\bar{x}_5,\bar{x}_6),
 (\bar{x}_0,\bar{x}_1,\bar{x}_2,\bar{x}_5,\bar{x}_6),(\bar{x}_0,\bar{x}_1,\bar{x}_2,\bar{x}_4,\bar{x}_6),\\
 (\bar{x}_0,\bar{x}_1,\bar{x}_2,\bar{x}_6,\bar{x}_3+b\bar{x}_4+\bar{x}_5),
 (\bar{x}_0,\bar{x}_1,\bar{x}_2,\bar{x}_4,\bar{x}_5,\bar{x}_3+\frac{1}{a}\bar{x}_6), \\
  (\bar{x}_0,\bar{x}_1,\bar{x}_2,\bar{x}_3,\bar{x}_5,\bar{x}_4+\bar{x}_6),(\bar{x}_0,\bar{x}_1,\bar{x}_3,\bar{x}_4,\bar{x}_5,\bar{x}_2+\bar{x}_6), \\
   (\bar{x}_0,\bar{x}_1,\bar{x}_3,\bar{x}_4,\bar{x}_5,\bar{x}_6), (\bar{x}_0,\bar{x}_2,\bar{x}_3,\bar{x}_4,\bar{x}_5,\bar{x}_6) \\
      (\bar{x}_0,\bar{x}_2,\bar{x}_3,\bar{x}_4,\bar{x}_6,\bar{x}_1+\bar{x}_5), (\bar{x}_0,\bar{x}_1,\bar{x}_2,\bar{x}_3,\bar{x}_6,\bar{x}_4+\frac{1}{b}\bar{x}_5) \\
    (\bar{x}_0,\bar{x}_1,\bar{x}_4,\bar{x}_5,\bar{x}_6,\bar{x}_2+a\bar{x}_3+x_6), (\bar{x}_0,\bar{x}_1,\bar{x}_2,\bar{x}_5,\bar{x}_6,\bar{x}_3+b\bar{x}_4),\\
      (\bar{x}_0,\bar{x}_1,\bar{x}_2,\bar{x}_4,\bar{x}_6,\bar{x}_3+\bar{x}_5),
      
      \mathfrak{m}_R
\end{cases}\]
The calculated colons are
\[\begin{cases} 
 (0_R):(\bar{x}_0) = (\bar{x}_2,\bar{x}_3,\bar{x}_0+\bar{x}_1+\bar{x}_4+\bar{x}_5+\bar{x}_6),\\
 (0_R):(\bar{x}_2) = (\bar{x}_0,\bar{x}_4,\bar{x}_5) \\
 (\bar{x}_0):(\bar{x}_4) = (\bar{x}_0,\bar{x}_1,\bar{x}_2,\bar{x}_3+b\bar{x}_4+\bar{x}_5+b\bar{x}_6),\\
 (\bar{x}_0):(\bar{x}_1) = (\bar{x}_0,\bar{x}_3,\bar{x}_4,\bar{x}_6) \\
 (\bar{x}_2):(\bar{x}_3) =  (\bar{x}_0,\bar{x}_1,\bar{x}_2,\bar{x}_3+b\bar{x}_4+\bar{x}_5+\frac{1}{a}\bar{x}_6), \\
 (\bar{x}_0):(\bar{x}_6) =  (\bar{x}_0,\bar{x}_1,\bar{x}_5,\bar{x}_2+a\bar{x}_3+\bar{x}_4+\bar{x}_6)\\
  (\bar{x}_2,\bar{x}_3):(\bar{x}_0)=(\bar{x}_2,\bar{x}_3,\bar{x}_0+\bar{x}_1+\bar{x}_4+\bar{x}_5+\bar{x}_6) \\
 (\bar{x}_2,\bar{x}_3):(\bar{x}_0+\bar{x}_1+\bar{x}_4+\bar{x}_5+\bar{x}_6)=(\bar{x}_0,\bar{x}_2,\bar{x}_3,\bar{x}_6,\bar{x}_4+\frac{1}{b}\bar{x}_5) \\
 (\bar{x}_0,\bar{x}_4):(\bar{x}_5) = (\bar{x}_0,\bar{x}_2,\bar{x}_4,\bar{x}_6,\bar{x}_1+\bar{x}_3+\bar{x}_5)\\
 (\bar{x}_0,\bar{x}_1):(\bar{x}_2) = (\bar{x}_0,\bar{x}_1,\bar{x}_4,\bar{x}_5)
      \end{cases}\]
 \[\begin{cases} 
 (\bar{x}_0,\bar{x}_4):(\bar{x}_3) = (\bar{x}_0,\bar{x}_1,\bar{x}_4,\bar{x}_2+a\bar{x}_3+a\bar{x}_5+\bar{x}_6)\\
 (\bar{x}_0,\bar{x}_1):(\bar{x}_4) = (\bar{x}_0,\bar{x}_1,\bar{x}_2,\bar{x}_3+b\bar{x}_4+\bar{x}_5+b\bar{x}_6)\\
 (\bar{x}_0,\bar{x}_4):(\bar{x}_2) = (\bar{x}_0,\bar{x}_4,\bar{x}_5)\\
 (\bar{x}_0,\bar{x}_6):(\bar{x}_2) = (\bar{x}_0,\bar{x}_4,\bar{x}_5,\bar{x}_6)
  (\bar{x}_0,\bar{x}_1):(\bar{x}_5) = (\bar{x}_0,\bar{x}_1,\bar{x}_2,\bar{x}_6,\bar{x}_3+b\bar{x}_4+\bar{x}_5)\\
  (\bar{x}_0,\bar{x}_1,\bar{x}_2):(\bar{x}_3+b\bar{x}_4+\bar{x}_5+b\bar{x}_6) = (\bar{x}_0,\bar{x}_1,\bar{x}_2,\bar{x}_4,\bar{x}_5,\bar{x}_3+\frac{1}{a}\bar{x}_6)\\
  (\bar{x}_0,\bar{x}_1,\bar{x}_2):(\bar{x}_3+b\bar{x}_4+\bar{x}_5+\frac{1}{a}\bar{x}_6) =(\bar{x}_0,\bar{x}_1,\bar{x}_2,\bar{x}_3,\bar{x}_5,\bar{x}_4+\bar{x}_6)\\
  (\bar{x}_0,\bar{x}_3,\bar{x}_4):(\bar{x}_6)  = (\bar{x}_0,\bar{x}_1,\bar{x}_3,\bar{x}_4,\bar{x}_5,\bar{x}_2+\bar{x}_6)\\
  (\bar{x}_0,\bar{x}_1,\bar{x}_4):(\bar{x}_5) = (\bar{x}_0,\bar{x}_1,\bar{x}_2,\bar{x}_4,\bar{x}_6,\bar{x}_3+\bar{x}_5)\\
  (\bar{x}_0,\bar{x}_1,\bar{x}_4):(\bar{x}_2+a\bar{x}_3+a\bar{x}_5+\bar{x}_6) = (\bar{x}_0,\bar{x}_1,\bar{x}_3,\bar{x}_4,\bar{x}_5,\bar{x}_6)\\
  (\bar{x}_0,\bar{x}_2,\bar{x}_4):(\bar{x}_6)  = (\bar{x}_0,\bar{x}_1,\bar{x}_2,\bar{x}_4,\bar{x}_5,\bar{x}_3+\frac{1}{a}\bar{x}_6)\\
  (\bar{x}_0,\bar{x}_1,\bar{x}_2):(\bar{x}_6) = (\bar{x}_0,\bar{x}_1,\bar{x}_2,\bar{x}_5,\bar{x}_3+\frac{1}{a}\bar{x}_4+\frac{1}{a}\bar{x}_6)\\
  (\bar{x}_0,\bar{x}_2,\bar{x}_3):(\bar{x}_6) =  (\bar{x}_0,\bar{x}_1,\bar{x}_2,\bar{x}_3,\bar{x}_5,\bar{x}_4+\bar{x}_6)\\
  (\bar{x}_0,\bar{x}_1,\bar{x}_5):(\bar{x}_6) =  (\bar{x}_0,\bar{x}_1,\bar{x}_5,\bar{x}_2+a\bar{x}_3+\bar{x}_4+\bar{x}_6)\\
  (\bar{x}_0,\bar{x}_1,\bar{x}_2):(\bar{x}_5) =  (\bar{x}_0,\bar{x}_1,\bar{x}_2,\bar{x}_6,\bar{x}_3+b\bar{x}_4+\bar{x}_5)\\
  (\bar{x}_0,\bar{x}_4,\bar{x}_5):(\bar{x}_6) =  (\bar{x}_0,\bar{x}_1,\bar{x}_4,\bar{x}_5,\bar{x}_2+a\bar{x}_3+\bar{x}_6)\\
  (\bar{x}_0,\bar{x}_1,\bar{x}_5):(\bar{x}_2+a\bar{x}_3+\bar{x}_4+\bar{x}_6) =  (\bar{x}_0,\bar{x}_1,\bar{x}_5,\bar{x}_6,\bar{x}_3+b\bar{x}_4)\\
  (\bar{x}_0,\bar{x}_1,\bar{x}_2,\bar{x}_5):(\bar{x}_3+b\bar{x}_4+b\bar{x}_6) = (\bar{x}_0,\bar{x}_1,\bar{x}_2,\bar{x}_4,\bar{x}_5,\bar{x}_3+\frac{1}{a}\bar{x}_6) \\
  (\bar{x}_0,\bar{x}_2,\bar{x}_4,\bar{x}_6):(\bar{x}_1+\bar{x}_3+\bar{x}_5) = (\bar{x}_0,\bar{x}_2,\bar{x}_3,\bar{x}_4,\bar{x}_5,\bar{x}_6)\\
  (\bar{x}_0,\bar{x}_2,\bar{x}_3,\bar{x}_6):(\bar{x}_4+\frac{1}{b}\bar{x}_5 ) = (\bar{x}_0,\bar{x}_2,\bar{x}_3,\bar{x}_4,\bar{x}_6,\bar{x}_1+\bar{x}_5)\\
  (\bar{x}_0,\bar{x}_3,\bar{x}_4,\bar{x}_6):(\bar{x}_5 ) = (\bar{x}_0,\bar{x}_2,\bar{x}_3,\bar{x}_4,\bar{x}_6,\bar{x}_1+\bar{x}_5)\\
  (\bar{x}_0,\bar{x}_1,\bar{x}_2,\bar{x}_6):(\bar{x}_3+b\bar{x}_4+\bar{x}_5) = \mathfrak{m}_R\\
  (\bar{x}_0,\bar{x}_1,\bar{x}_2,\bar{x}_5):(\bar{x}_4 ) = (\bar{x}_0,\bar{x}_1,\bar{x}_2,\bar{x}_5,\bar{x}_3+b\bar{x}_4+bx_6)\\
  (\bar{x}_0,\bar{x}_2,\bar{x}_3,\bar{x}_6):(\bar{x}_4 ) = (\bar{x}_0,\bar{x}_1,\bar{x}_2,\bar{x}_3,\bar{x}_6,\bar{x}_4+\frac{1}{b}\bar{x}_5)\\
  (\bar{x}_0,\bar{x}_1,\bar{x}_2,\bar{x}_5):(\bar{x}_3 ) = (\bar{x}_0,\bar{x}_1,\bar{x}_2,\bar{x}_5,\bar{x}_3+b\bar{x}_4+\frac{1}{a}\bar{x}_6)\\
  (\bar{x}_0,\bar{x}_1,\bar{x}_2,\bar{x}_5):(\bar{x}_3+\frac{1}{a}\bar{x}_4+\frac{1}{a}\bar{x}_6)=(\bar{x}_0,\bar{x}_1,\bar{x}_2,\bar{x}_5,\bar{x}_6,\bar{x}_3+b\bar{x}_4) \\
  (\bar{x}_0,\bar{x}_1,\bar{x}_2,\bar{x}_5):(\bar{x}_3+b\bar{x}_4+\frac{1}{a}\bar{x}_6) = 
  (\bar{x}_0,\bar{x}_1,\bar{x}_2,\bar{x}_3,\bar{x}_5,\bar{x}_4+\bar{x}_6)\\
  (\bar{x}_0,\bar{x}_1,\bar{x}_4,\bar{x}_5):(\bar{x}_3 ) = (\bar{x}_0,\bar{x}_1,\bar{x}_4,\bar{x}_5,\bar{x}_2+a\bar{x}_3+\bar{x}_6)\\
  (\bar{x}_0,\bar{x}_1,\bar{x}_2,\bar{x}_6):(\bar{x}_3+b\bar{x}_4) = (\bar{x}_0,\bar{x}_1,\bar{x}_2,\bar{x}_6,\bar{x}_3+b\bar{x}_4+\bar{x}_5)\\
  (\bar{x}_0,\bar{x}_1,\bar{x}_4,\bar{x}_5):(\bar{x}_2+a\bar{x}_3+\bar{x}_6) = (\bar{x}_0,\bar{x}_1,\bar{x}_3,\bar{x}_4,\bar{x}_5,\bar{x}_6)\\
  (\bar{x}_0,\bar{x}_1,\bar{x}_5,\bar{x}_6):(\bar{x}_3 +b\bar{x}_4) = (\bar{x}_0,\bar{x}_1,\bar{x}_4,\bar{x}_5,\bar{x}_6,\bar{x}_2+a\bar{x}_3)\\
  (\bar{x}_0,\bar{x}_1,\bar{x}_4,\bar{x}_5):(\bar{x}_6) = (\bar{x}_0,\bar{x}_1,\bar{x}_4,\bar{x}_5,\bar{x}_2+a\bar{x}_3+\bar{x}_6)\\
  (\bar{x}_0,\bar{x}_1,\bar{x}_2,\bar{x}_5):(\bar{x}_6) = (\bar{x}_0,\bar{x}_1,\bar{x}_2,\bar{x}_5,\bar{x}_3+\frac{1}{a}\bar{x}_4+\frac{1}{a}\bar{x}_6)\\
  (\bar{x}_0,\bar{x}_1,\bar{x}_2,\bar{x}_6):(\bar{x}_4 ) = (\bar{x}_0,\bar{x}_1,\bar{x}_2,\bar{x}_6,\bar{x}_3+b\bar{x}_4+\bar{x}_5)\\
  (\bar{x}_0,\bar{x}_1,\bar{x}_2,\bar{x}_6):(\bar{x}_3+b\bar{x}_4+\bar{x}_5) = \mathfrak{m}_R\\
  (\bar{x}_0,\bar{x}_1,\bar{x}_2,\bar{x}_4,\bar{x}_5):(\bar{x}_3+\frac{1}{a}\bar{x}_6) = \mathfrak{m}_R\\
  (\bar{x}_0,\bar{x}_1,\bar{x}_2,\bar{x}_3,\bar{x}_5):(\bar{x}_4+\bar{x}_6) = \mathfrak{m}_R\\
  (\bar{x}_0,\bar{x}_1,\bar{x}_3,\bar{x}_4,\bar{x}_5):(\bar{x}_2+\bar{x}_6) = (\bar{x}_0,\bar{x}_1,\bar{x}_3,\bar{x}_4,\bar{x}_5,\bar{x}_6)\\
              \end{cases}\]
 \[\begin{cases} 
  (\bar{x}_0,\bar{x}_3,\bar{x}_4,\bar{x}_5,\bar{x}_6):(\bar{x}_1) = (\bar{x}_0,\bar{x}_3,\bar{x}_4,\bar{x}_5,\bar{x}_6)\\
  (\bar{x}_0,\bar{x}_3,\bar{x}_4,\bar{x}_5,\bar{x}_6):(\bar{x}_2) = (\bar{x}_0,\bar{x}_3,\bar{x}_4,\bar{x}_5,\bar{x}_6)\\
  (\bar{x}_0,\bar{x}_2,\bar{x}_3,\bar{x}_4,\bar{x}_6):(\bar{x}_1+\bar{x}_5) = (\bar{x}_0,\bar{x}_2,\bar{x}_3,\bar{x}_4,\bar{x}_5,\bar{x}_6)\\
(\bar{x}_0,\bar{x}_2,\bar{x}_3,\bar{x}_6,\bar{x}_4+\frac{1}{b}x_5):(\bar{x}_1)= (\bar{x}_0,\bar{x}_2,\bar{x}_3,\bar{x}_4,\bar{x}_5,\bar{x}_6)\\ 
    (\bar{x}_0,\bar{x}_1,\bar{x}_4,\bar{x}_5,\bar{x}_2+a\bar{x}_3+\bar{x}_6):(\bar{x}_6)= (\bar{x}_0,\bar{x}_1,\bar{x}_4,\bar{x}_5,\bar{x}_2+a\bar{x}_3+\bar{x}_6)\\ 

  (\bar{x}_0,\bar{x}_1,\bar{x}_5,\bar{x}_6,\bar{x}_3+b\bar{x}_4):(\bar{x}_2) = (\bar{x}_0,\bar{x}_1,\bar{x}_3,\bar{x}_4,\bar{x}_5,\bar{x}_6)\\
  (\bar{x}_0,\bar{x}_1,\bar{x}_2,\bar{x}_4,\bar{x}_6):(\bar{x}_3+\bar{x}_5) = \mathfrak{m}_R\\
 (\bar{x}_0,\bar{x}_1,\bar{x}_3,\bar{x}_4,\bar{x}_5,\bar{x}_6):(\bar{x}_2) = (\bar{x}_0,\bar{x}_1,\bar{x}_3,\bar{x}_4,\bar{x}_5,\bar{x}_6) \\ 
                 \end{cases}\] 
    \indent Note that $|\mathcal{F}|=|57|.$ Every colon is a non-trivial calculation. We prove one of the equalities to demonstrate this, and direct the reader to code to verify the other equalities. Let $J$ be the defining ideal for $R.$ We prove $(\bar{x}_0):(\bar{x}_1) = (\bar{x}_0,\bar{x}_3,\bar{x}_4,\bar{x}_6).$ To prove this, we must first show the equality
\[\left(J+(x_0)\right) = L_1 \cap L_2 \cap L_3 \cap (L_4 + (x_0)) \cap (L_5 + (x_0)).\] To prove the equality we will need $H_{S/J}(t).$ \\
\indent The smallest non-negative integer $\alpha$ satisfying $\binom{6+\alpha}{\alpha} \geq 5(\alpha+1)$  is $2,$ so by Theorem \ref{thm:reg}, $\mathrm{reg}_S(S/J)=2.$ Thus, by Lemma \ref{lem:hilbertSeries}
    \begin{align*}  
        H_{S/J}(t) &= \frac{-3t^3+2t^2+5t+1}{(1-t)^2}.
    \end{align*}
\noindent Label 
\[I = (L_1 \cap L_2 \cap L_3 )+ (L_4 + (x_0)) \cap (L_5 + (x_0))\]
and 
\[K= (L_1 \cap L_2 \cap L_3 \cap (L_4 + (x_0)) \cap (L_5 + (x_0)).\]
We have the inclusion $\left(J+(x_0)\right) \subseteq K.$ Using the short exact sequence
\begin{align*}
    0 \rightarrow S/(J:(x_0))[-1] \rightarrow S/J \rightarrow S/(J+(x_0)) \rightarrow 0,
\end{align*} and the additivity of the Hilbert series
    \begin{align*}
        H_{S/(J+(x_0))}(t) &= H_{S/J}(t) - tH_{S/J:(x_0)}(t)  \\
                         &= H_{S/J}(t) - tH_{S/(L_4 \cap L_5)}(t)  \\
                                  &=  \frac{-3t^3+2t^2+5t+1}{(1-t)^2} -  t\left( \frac{-t^2+2t+1}{(1-t)^2}\right) \\
                                  &= \frac{-2t^3+4t+1}{(1-t)^2}.
    \end{align*}
\noindent To finish proving the desired equality we must determine $H_{K}(t).$ To do so we calculate $H_{S/I}(t),$ $H_{S/(L_1 \cap L_2 \cap L_3 )}(t),$ and $H_{S/((L_4 + (x_0)) \cap (L_5 + (x_0))}(t)$ and then use the short exact sequence
\begin{align*}
    0 &\rightarrow S/K \rightarrow S/L_1 \cap L_2 \cap L_3  \oplus S/(L_4 + (x_0)) \cap (L_5 + (x_0))\rightarrow S/I \rightarrow 0
\end{align*}
to calculate $H_K(t).$\par
\indent  We first determine $H_{S/I}(t);$ notice the following intersection
\[(L_4 + (x_0)) \cap (L_5 + (x_0)) = (x_0,x_2,x_3,x_4+x_6,x_1+x_5,x_5x_6).\]
The set $\{x_0,x_2,x_3,x_4+x_6,x_1+x_5,x_4,x_6+x_5\}$ forms a basis of $S_1.$ We aim to show $x_4^2$ and $(x_5+x_6)^2$ vanish in $S/I.$ To this end, observe the following relations
\begin{align*}
     x_4x_6 &\in L_1 \cap L_2 \cap L_3,\\
     x_4+x_6 &\in (L_4 + (x_0)) \cap (L_5 + (x_0)),\\
     x_3x_4+bx_4^2+x_4x_5 &\in  L_1 \cap L_2 \cap L_3, \\
     x_3x_5+bx_4x_5+x_5^2 &\in  L_1 \cap L_2 \cap L_3, \\
     x_5x_6 &\in  L_1 \cap L_2 \cap L_3.
\end{align*}
The first two relations guarantee that $x_4^2$ and $x_6^2$ vanish in $S/I.$ The third relation guarantees that $x_4x_5$ vanishes in $S/I$, since $x_3$ and $x_4^2$ vanish in $S/I.$ The previous conclusions and the fourth relation guarantee that $x_5^2$ vanishes in $S/I.$ 
All the previous conclusions and the last relation guarantee that $(x_5+x_6)^2$ vanishes in $S/I.$ Thus,
\[ H_{S/I}(t) = 1+2t.\]
\indent Now, Lemma $\ref{lem:hilbertSeries}$ and Theorem $\ref{thm:hsForPoints}$ yield the two Hilbert series
\begin{align*}
    H_{S/(L_1 \cap L_2 \cap L_3)}(t) &= \frac{-2t^2+4t+1}{(1-t)^2},\\
    H_{S/((L_4+(x_0)) \cap (L_5 +(x_0)))}(t) &= \frac{t+1}{1-t}.
\end{align*}
By the additivity of the Hilbert series
\begin{align*}
    H_{ S/K}(t) &= H_{ S/L_1 \cap L_2 \cap L_3}(t) +  H_{ S/L_4 + (x_0) \cap L_5 + (x_0)}(t) - H_{S/I}(t)  \\
    &= \frac{-2t^2+4t+1}{(1-t)^2} + \frac{t+1}{1-t} - (1+2t), \\
    &= \frac{-2t^3+4t+1}{(1-t)^2}.
\end{align*}
Thus, we have proven the desired equality. \par
\indent We can now show $(\bar{x}_0):(\bar{x}_1) = (\bar{x}_0,\bar{x}_3,\bar{x}_4,\bar{x}_6).$
The inclusion 
\[(\bar{x}_0,\bar{x}_3,\bar{x}_4,\bar{x}_6) \subseteq (\bar{x}_0):(\bar{x}_1),\]
is immediate. We aim to show $H_{S/((J+(x_0)):(x_1))}(t)=H_{S/(J+(x_0,x_3,x_4,x_5))}(t).$ To begin, we calculate $H_{S/((J+(x_0)):(x_1))}(t)$. Our previous equality
\[\left(J+(x_0)\right) = L_1 \cap L_2 \cap L_3 \cap (L_4 + (x_0)) \cap (L_5 + (x_0)),\]
yields the following
\[(J+(x_0)):(x_1) = K:(x_1) = L_1 \cap (L_5+(x_0)).\]
Using the short exact sequence
\[0 \rightarrow S/(L_1 \cap (L_5+(x_0))) \rightarrow S/L_1 \oplus S/(L_5+(x_0)) \rightarrow S/(L_1+L_5+(x_0)) \rightarrow 0\]
and the additivity of the Hilbert series
\begin{align*}
    H_{ S/((J+(x_0)):(x_1))}(t) &= H_{S/(L_1 \cap (L_5+(x_0)))}(t) \\
    &= H_{S/L_1}(t) + H_{S/(L_5+(x_0))}(t) - H_{S/(L_1+L_5+(x_0))} \\
    &= \frac{1}{(1-t)^2}+\frac{1}{1-t}-1 \\
    &=\frac{-t^2+t+1}{(1-t)^2}.
\end{align*} \par
\indent We now need to determine  $H_{R/(\bar{x}_0,\bar{x}_3,\bar{x}_4,\bar{x}_6)}(t).$ Consider the intersection 
\[L_1 \cap \left(L_5+(x_0)\right) = (x_0,x_3,x_4,x_6,x_2x_5,x_5x_1+x_5^2).\] 
We have the relations
\begin{align*}
    x_0x_5+x_1x_5+x_3x_5+bx_4x_5+x_5^2 &\in J, \\
    x_2x_5 &\in J.
\end{align*}
Using these relations yields
\begin{align*}
     H_{S/(J+L_1 \cap \left(L_5+(x_0)\right)}(t) &= H_{S/(J+(x_0,x_3,x_4,x_6,x_2x_5,x_1x_5+x_5^2))}(t) \\
                                                 &= H_{S/(J+(x_0,x_3,x_4,x_6))}(t). 
     \end{align*}
We can calculate $H_{S/(J+L_1 \cap \left(L_5+(x_0)\right)}(t)$ using the short exact sequence
\begin{align}\label{ses:J++}
    0 \rightarrow S/(J \cap L_1 \cap (L_5+(x_0))) &\rightarrow S/J \oplus S/ (L_1 \cap (L_5+(x_0))) \\ 
 \nonumber   &\rightarrow S/(J  + L_1 \cap (L_5+(x_0)) \rightarrow 0,
\end{align}
we only need $H_{S/(J \cap L_1 \cap (L_5+(x_0))}(t).$ By the modular law
\[
    J \cap (L_5+(x_0)) = \left( \bigcap_{i=2}^{5}L_i\right) \cap \left( L_1 \cap L_5 +L_1 \right) 
                       = \left( \bigcap_{i=2}^{5}L_i\right) \cap L_1 
                       =J.
\]
So, $H_{S/(J \cap L_1 \cap (L_5+(x_0)))}(t) = H_{S/J}(t).$ The additivity of the Hilbert series on short exact sequence (\ref{ses:J++}) yields
\begin{align*}
    H_{S/(J + (L_1 \cap (L_5+(x_0)))}(t) &= H_{J}(t) + H_{S/ (L_1 \cap (L_5+(x_0)))}(t) - H_{S/(J \cap L_1 \cap (L_5+(x_0)))}(t) \\
                                   &= H_{J}(t) + H_{S/ (L_1 \cap (L_5+(x_0)))}(t) - H_{S/J}(t) \\
                                   &= H_{S/ (L_1 \cap (L_5+(x_0)))}(t)\\
                                   &= \frac{-t^2+t+1}{(1-t)^2}.
\end{align*}
So, $H_{S/(J+(x_0,x_3,x_4,x_6))}(t) = H_{S/((J+(x_0)):(x_1))}(t),$ which proves the claim. Every other colon is calculated similarly, and requires identical arguments.
\end{proof}
For the interested reader, there is a Macaulay2 file you may run verifying the equalities located at  www.joshuaandrewrice.com.

\section{Hilbert function obstruction to the Koszul property}\label{section:negative}
In this section, we determine when the coordinate ring of a generic collection of lines is not Koszul. But first, we need a theorem from Complex Analysis.
\begin{thm}[Vivanti–Pringsheim, {\cite[Chapter 8, Section 1]{remmert1991theory}} ] \label{thm:vp}
Let the power series $f(z) = \sum a_v z^v$ have positive finite radius of convergence $r$ and suppose that all but finitely many of its coefficients $a_v$ are real and non-negative. Then $z = r$ is a singular point of $f(z)$.
\end{thm}
\begin{thm}\label{thm:notKoszul}
Let $\mathcal{M}$ be a generic collection of $m$ lines in $\mathbb{P}^n$ with $n \geq 2$ and $R$ the coordinate ring of $\mathcal{M}.$ If 
$$ m > \frac{1}{72}\left(3(n^2+10n+13)+\sqrt{3(n-1)^3(3n+5)}\right),$$
then $R$ is not Koszul.
\end{thm}
\begin{proof}
    We prove the claim by contradiction. Suppose that $\mathrm{reg}_S(R) = \alpha.$ Note that by Theorem \ref{thm:reg}, $\alpha$ is the smallest non-negative integer such that $\binom{n+\alpha}{\alpha} \geq m(\alpha+1).$ We have four cases: $\alpha=0,$ $\alpha = 1,$ $\alpha =2,$ or $\alpha \geq 3.$
\begin{enumerate}
\item Suppose that $\alpha = 0.$ Then 
\[1 <  \frac{1}{72}\left(3(n^2+10n+13)+\sqrt{3(n-1)^3(3n+5)}\right) < m \leq 1,\]
a contradiction.
    \item If $\alpha = 1,$ then $2m \leq n+1,$ and hence
    \begin{align*}
        m &\leq \frac{n+1}{2} < \frac{1}{72}\bigg(3(n^2+10n+13)+\sqrt{3(n-1)^3(3n+5)}\bigg)< m, \hspace{-.8cm}
    \end{align*}
    a contradiction.
    \item Now assume that $\alpha=2$ and that $R$ is Koszul. By Lemma \ref{lem:hilbertSeries}, the Hilbert series for $R$ is
\begin{align*}
    H_R(t) &= \frac{(n+1-2m)t^3+(3m-2n-1)t^2+(n-1)t +1}{ (1-t)^2}. 
\end{align*}
Thus, by Equation (\ref{eqn:froberg})
$$ \text{P}_{\mathbb{C}}^R(t) =\frac{1}{\text{H}_R(-t)}=\tfrac{(1+t)^2}{ (2m-n-1)t^3+(3m-2n-1)t^2 +(1-n)t+1}.$$
 Denote \[p(t) =1+(1-n)t+(3m-2n-1)t^2+(2m-n-1)t^3\] and note the leading coefficient is positive, since $n+1 < 2m$. 
 By the Intermediate Value Theorem $p(t)$ has a negative zero, since $p(0)=1$ and 
 \[ p(-3)=-27m+12n+16 < 0,\]
 since $n+1 < 2m$ and $1<m.$ So, the radius of convergence $r$ of $P_{\mathbb{C}}^R(t)$ is finite and all the coefficients are positive. So, by Theorem \ref{thm:vp}, $r$ must occur as a singular point of $P_{\mathbb{C}}^R(t)$; meaning that $p(t)$ must have $3$ real roots and one of them must be positive. Recall that if the discriminant of a cubic polynomial with real coefficients is negative, then the polynomial has $2$ non-real complex roots. Thus, the discriminant of $p(t)$ must be non-negative. The discriminant of $p(t)$ is
\[ \Delta =-m (108 m^2 - 9 m (n^2 + 10 n + 13) + 4 (n + 2)^3).\]
We view the discriminate as a continuous function of $m.$ Now, note that the leading term of $\Delta$ is negative. Applying the quadratic formula to the quadratic term above and only considering the larger root of the two yields the following
\begin{align*}
        m &= \dfrac{9(n^2+10n+13) + \sqrt{9^2(n^2+10n+13)^2-4(108)(4)(n+2)^3}}{2(108)} \hspace{-1.3cm} \\
         &=\frac{ 3(n^2+10n+13) +\sqrt{9n^4 - 12n^3 - 18n^2 + 36n - 15}}{72} \\
         &=\frac{ 3(n^2+10n+13) +\sqrt{3(n-1)^3(3n+5)}}{72}.
\end{align*}
Since, we have a unique positive root in the quadratic term and $m > 0,$ we may conclude that
    $$m \leq \frac{1}{72}\bigg(3(n^2+10n+13)+\sqrt{3(n-1)^3(3n+5)}\bigg),$$
    a contradiction.
    \item Suppose that $\alpha \geq 3$ and $R$ is Koszul. By Theorem \ref{thm:hh}, the defining ideal of $R$ contains a form of degree $\alpha$ in a minimal generating  set, where $\alpha \geq 3$. Thus, $R$ is not quadratic, a contradiction.
    \end{enumerate}
    Hence, $R$ is not Koszul.
\end{proof}
We have at least one exceptional example of a coordinate ring of a generic collection of lines that is not Koszul that the previous theorem does not handle.  
\begin{prop}\label{prop:exampleNotKoszul}
Let $\mathcal{M}$ be a collection of $3$ lines in general linear position in $\mathbb{P}^4$ and $R$ the coordinate ring of $\mathcal{M}.$ The defining ideal $J$ for $R$ has a cubic in a minimal generating set. Hence, $R$ is not Koszul.
\end{prop}
\begin{proof}
By Remark \ref{rmk:linesIntersection} and a change of basis, we may assume the defining ideals for our three lines have the form
\begin{tasks}(3)
\task[]   $L_1 = (x_0,x_1,x_3),$
\task[]   $L_2 = (x_0,x_2,x_4),$
\task[]   $L_3 = (x_1,x_2,l),$
\end{tasks}
where $l=x_3+x_4.$ Let $J$ be the defining ideal for $R$ and notice that $K = L_1 \cap L_2=(x_0,x_1x_2,x_1x_4,x_3x_2,x_3x_4).$ \par 
\indent We have the following ring isomorphism
\begin{align*}
    S/(K+L_3) &= \mathbb{C}[x_0,x_1,x_2,x_3,x_4]/(x_0,x_1,x_2,l,x_3x_4) \\ 
                           &\cong \mathbb{C}[x_3,x_4]/(l,x_3x_4) \\
                           &\cong \mathbb{C}[w]/(w^2).
    \end{align*}
Hence,
\[H_{S/(K+L_3) }(t) = \dfrac{-t^2+1}{1-t} = 1+t.\]
Therefore, by Proposition \ref{prop:regInequality} the $\text{reg}_S(S/(K+L_3)) =1.$ \par
\indent 
One checks that the $\text{reg}_S(S/K) =1.$ Using the short exact sequence
\[0 \rightarrow S/J \rightarrow S/K\oplus S/L_3 \rightarrow S/(K +L_3) \rightarrow 0 \]
and Proposition \ref{prop:regInequality} yields $\text{reg}(S/J) \leq 2.$ So $J$ is generated by forms of degree at most $3$. The previous short exact sequence, Lemma \ref{lem:hilbertSeries}, and the additivity of the Hilbert series along the previous short exact sequence yields
\begin{align*}
    H_{S/J}(t) &= H_{S/K}(t) + H_{S/L_3}(t) - H_{S/(K+L_3)}(t) \\
                         &= \frac{-t^2+2t+1}{(1-t)^2} + \frac{1}{(1-t)^2} - (1+t) \\
                         &= \frac{-t^3+3t+1}{(1-t)^2} \\
                         &= 1+5t+9t^2+12t^3 + \cdots.
\end{align*}
\indent Thus, $J$ is generated by $6$ linearly independent quadrics and possibly cubics. The cubic $x_3x_4l$ is contained in $J$, but is not contained in the ideal $(x_0x_1,x_0x_2,x_1x_2,x_1x_4,x_2x_3,x_0l),$ since no term divides $x_3^2x_4.$ Hence, there must be a cubic generator in a minimal generating set of $J.$ Thus, $R$ is not Koszul.
\end{proof}
\begin{rmk}
 Since Remark \ref{rmk:linesIntersection} says that a generic collection of lines is in general linear position, then we may use Lemma \ref{lem:hilbertSeries} to show that the coordinate ring of a generic collection of $3$ lines in $\mathbb{P}^4$ has the same Hilbert series as $R.$ 
 \end{rmk}
\section{Examples}\label{section:examples}
Finally it is worth observing $3$ examples that have appeared while studying generic lines. 
    \begin{eg}  There are collections of lines in general linear position that are not generic collections. Consider the four lines in $\mathbb{P}^3$:      
    \begin{align*} 
        \mathcal{L}_1 &= \{[0:0:\alpha:\beta] : \alpha, \beta \textrm{ not both zero} \} \\
        \mathcal{L}_2 &= \{[\alpha:\beta:0:0] : \alpha, \beta \textrm{ not both zero} \}, \\ \mathcal{L}_3 &=\{ [\alpha : \beta : -\alpha : \beta ] : \alpha, \beta \textrm{ not both zero} \}, \\
        \mathcal{L}_4 &= \{[\alpha:-\beta:\alpha:\beta] : \alpha, \beta \textrm{ not both zero}\}.
        \end{align*}
 These lines are in general linear position since every pair spans $\mathbb{P}^3.$  
The four defining ideals in $S$ are
    \begin{align*} 
    L_1 &= (x_0,x_1), 
    \hspace{3cm} L_2 = (x_2,x_3),  \\
    L_3 &= (x_0+x_2,x_1-x_3),
     \hspace{1.3cm} L_4 = (x_0-x_2,x_1+x_3).        
    \end{align*}
 The coordinate ring $S/J,$ where $J = \bigcap_{i=1}^4 L_i,$ has the following Hilbert series 
$$ H_{S/J}(t) = \frac{-3t^4+2t^3+2t^2+2t+1}{(1-t)^2} = 1+4t+9t^2+\cdots,$$
whereas, by Theorem \ref{thm:hh}, the coordinate ring $R$ for $4$ generic lines in $\mathbb{P}^3$ has the following Hilbert series: 
$$ H_R(t) = \frac{-2t^3+3t^2+2t+1}{(1-t)^2}= 1+4t+10t^2 + \cdots.$$
So, this is not a generic collection of lines. 
\end{eg}
    \begin{eg}  Consider the coordinate ring $R$ for  $5$ generic lines in $\mathbb{P}^5.$ The defining ideal $J$ for $R$ is minimally generated by quadrics and has the following Betti table computed via Macaulay2.
$$
 \begin{tikzpicture}
 
 \draw[thick,double](1,-4)--(1,0);
\draw[thick,double](0,-1)--(7,-1);

\draw[step=1cm,black,very thin] (7,-4) grid (0,0);
\foreach \x in {0,1,2,3,4,5}
\node at (1+\x+.5,-.5) {$\x$};
\foreach \x in {0,1,2}
\node at (.5,-1-\x-.5) {$\x$};

\node at (.5,-.5) {$S/J$};

\node at (1.5,-1.5) {$1$};
\node at (2.5,-1.5) {$-$};
\node at (3.5,-1.5) {$-$};
\node at (4.5,-1.5) {$-$};
\node at (5.5,-1.5) {$-$};
\node at (6.5,-1.5) {$-$};

\node at (1.5,-2.5) {$-$};
\node at (2.5,-2.5) {$6$};
\node at (3.5,-2.5) {$-$};
\node at (4.5,-2.5) {$-$};
\node at (5.5,-2.5) {$-$};
\node at (6.5,-2.5) {$-$};

\node at (1.5,-3.5) {$-$};
\node at (2.5,-3.5) {$-$};
\node at (3.5,-3.5) {$25$};
\node at (4.5,-3.5) {$36$};
\node at (5.5,-3.5) {$20$};
\node at (6.5,-3.5) {$4$};

\end{tikzpicture}$$

The ring $R$ is not Koszul by Theorem \ref{thm:notKoszul}. Furthermore,  it is known that if $R$ is Koszul and the defining ideal is generated by $g$ elements, then $\beta_{i,i+1} \leq \binom{g}{i}$ for $i \in \{2,\ldots,g\}$ {\cite[Proposition 2.3]{mantero2020betti}}. The previous inequality fails for $i=2$, since $\binom{6}{2} < 25$.  So, this ring is not Koszul for two numerical reasons. 
\end{eg} 
 \begin{eg} 
 Consider the coordinate ring $R$ for $6$ generic lines in $\mathbb{P}^6.$ The defining ideal $J$ for $R$ is minimally generated by quadrics and has the following Betti table computed via Macaulay2

$$\begin{tikzpicture}

\draw[thick,double](1,-4)--(1,0);
\draw[thick,double](0,-1)--(8,-1);
\draw[step=1cm,black,very thin] (8,-4) grid (0,0);
\foreach \x in {0,1,2,3,4,5,6}
\node at (1+\x+.5,-.5) {$\x$};
\foreach \x in {0,1,2}
\node at (.5,-1-\x-.5) {$\x$};

\node at (.5,-.5) {$S/J$};

\node at (1.5,-1.5) {$1$};
\node at (2.5,-1.5) {$-$};
\node at (3.5,-1.5) {$-$};
\node at (4.5,-1.5) {$-$};
\node at (5.5,-1.5) {$-$};
\node at (6.5,-1.5) {$-$};
\node at (7.5,-1.5) {$-$};

\node at (1.5,-2.5) {$-$};
\node at (2.5,-2.5) {$10$};
\node at (3.5,-2.5) {$10$};
\node at (4.5,-2.5) {$-$};
\node at (5.5,-2.5) {$-$};
\node at (6.5,-2.5) {$-$};
\node at (7.5,-2.5) {$-$};

\node at (1.5,-3.5) {$-$};
\node at (2.5,-3.5) {$-$};
\node at (3.5,-3.5) {$30$};
\node at (4.5,-3.5) {$76$};
\node at (5.5,-3.5) {$70$};
\node at (6.5,-3.5) {$30$};
\node at (7.5,-3.5) {$5$};
\end{tikzpicture}$$

The Algebra is not Koszul by Theorem \ref{thm:notKoszul}, but it does not fail the aforementioned inequality. 
\end{eg} 

Coordinate rings with defining ideals minimally generated by quadrics are not rare, but the previous two examples are interesting since both fail for identical reasons and one fails for an additional numerical reason. It would be interesting to determine sufficent reasons for why certain numerical conditions fail, and others do not. For example, why does $\beta_{i,i+1} \leq \binom{g}{i}$ fail in one of the previous rings, but not the other. \par  
\indent Furthermore, we would like to add that our theorems do not cover every coordinate ring $R$ for every generic collection of lines in $\mathbb{P}^n.$ For the coordinate rings we could not determine, there is a possibility these rings could be LG-quadratic or G-quadratic. Meaning that in every possible case that is computable by Macaulay2 there exists a quadratic monomial ideal whose quotient ring gives the same Hilbert series as $R.$ There could even be some change of basis which gives a quadratic Gr\"{o}bner basis. Further, if we wanted to construct a Koszul filtration in these coordinate rings, then Proposition \ref{prop:exampleKoszul} demonstrates that there is no reason why we should expect a reasonable filtration, unless there is a more efficient change of basis that went unobserved. Below is a table, without $m=1,$ summarizing our results: \par

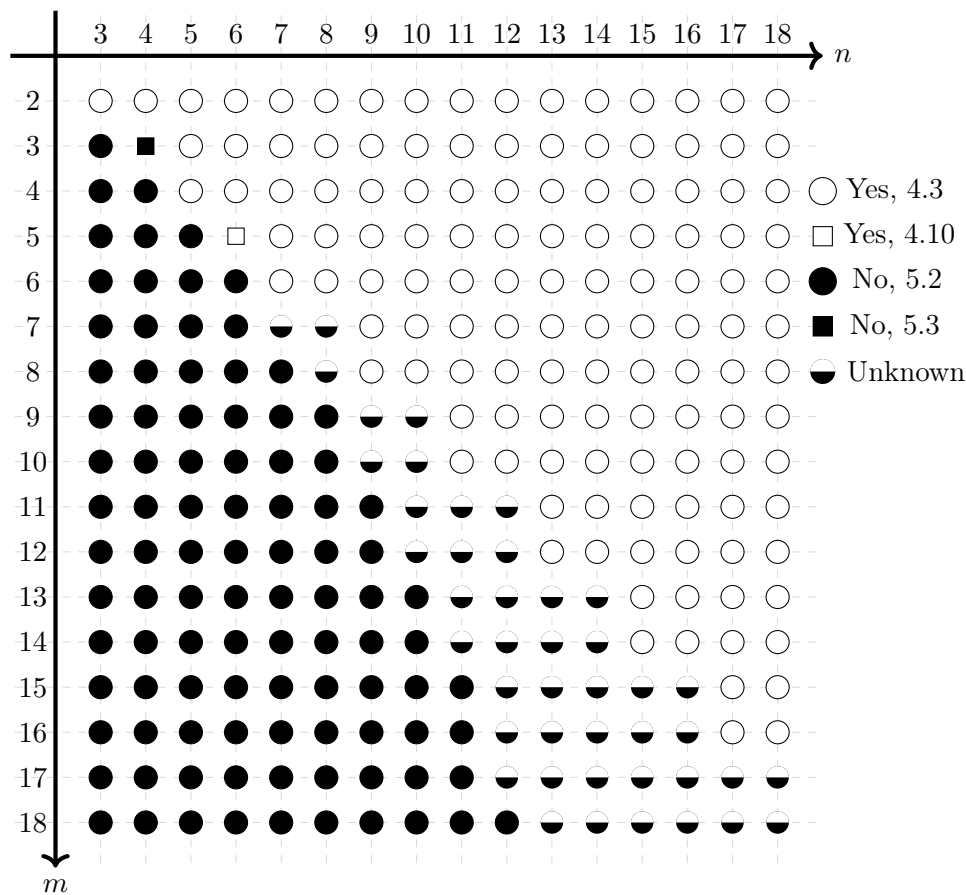
\begin{figure}[!htb]

 \flushleft  \begin{minipage}{.3\textwidth}
\begin{tikzpicture}[scale = .6]
\draw[help lines, color=gray!30, dashed] (-9.9,-11.9) grid (7.9,6.9);
\draw[->,ultra thick] (-10,6)--(8,6) node[right]{$n$};
\draw[->,ultra thick] (-9,7)--(-9,-12) node[below]{$m$};

\node at (-8,6.5) {$3$};  
\node at (-7,6.5) {$4$};  
\node at (-6,6.5) {$5$};  
\node at (-5,6.5) {$6$};  
\node at (-4,6.5) {$7$};  
\node at (-3,6.5) {$8$};  
\node at (-2,6.5) {$9$};  
\node at (-1,6.5) {$10$};  
\node at (0,6.5) {$11$};  
\node at (1,6.5) {$12$};  
\node at (2,6.5) {$13$};  
\node at (3,6.5) {$14$};  
\node at (4,6.5) {$15$};  
\node at (5,6.5) {$16$};  
\node at (6,6.5) {$17$};  
\node at (7,6.5) {$18$};  

\node at (-9.5,5) {$2$};  
\node at (-9.5,4) {$3$};  
\node at (-9.5,3) {$4$};
\node at (-9.5,2) {$5$};  
\node at (-9.5,1) {$6$};  
\node at (-9.5,0) {$7$};
\node at (-9.5,-1) {$8$};  
\node at (-9.5,-2) {$9$};  
\node at (-9.5,-3) {$10$};  
\node at (-9.5,-4) {$11$};  
\node at (-9.5,-5) {$12$};  
\node at (-9.5,-6) {$13$};  
\node at (-9.5,-7) {$14$};  
\node at (-9.5,-8) {$15$};  
\node at (-9.5,-9) {$16$};  
\node at (-9.5,-10) {$17$};  
\node at (-9.5,-11) {$18$};  

\node[circle,draw,fill = white,scale=.85]  (c) at (-8,5){};
\node[circle,draw,fill = black,scale=.85]  (c) at (-8,4){};
\node[circle,draw,fill = black,scale=.85]  (c) at (-8,3){};
\node[circle,draw,fill = black,scale=.85]  (c) at (-8,2){};
\node[circle,draw,fill = black,scale=.85]  (c) at (-8,1){};
\node[circle,draw,fill = black,scale=.85]  (c) at (-8,0){};
\node[circle,draw,fill = black,scale=.85]  (c) at (-8,-1){};
\node[circle,draw,fill = black,scale=.85]  (c) at (-8,-2){};
\node[circle,draw,fill = black,scale=.85]  (c) at (-8,-3){};
\node[circle,draw,fill = black,scale=.85]  (c) at (-8,-4){};
\node[circle,draw,fill = black,scale=.85]  (c) at (-8,-5){};
\node[circle,draw,fill = black,scale=.85]  (c) at (-8,-6){};
\node[circle,draw,fill = black,scale=.85]  (c) at (-8,-7){};
\node[circle,draw,fill = black,scale=.85]  (c) at (-8,-8){};
\node[circle,draw,fill = black,scale=.85]  (c) at (-8,-9){};
\node[circle,draw,fill = black,scale=.85]  (c) at (-8,-10){};
\node[circle,draw,fill = black,scale=.85]  (c) at (-8,-11){};

\node[circle,draw,fill = white,scale=.85]  (c) at (-7,5){};
\node[rectangle,draw, fill = black,scale=.85]  (c) at (-7,4){};
\node[circle,draw,fill = black,scale=.85]  (c) at (-7,3){};
\node[circle,draw,fill = black,scale=.85]  (c) at (-7,2){};
\node[circle,draw,fill = black,scale=.85]  (c) at (-7,1){};
\node[circle,draw,fill = black,scale=.85]  (c) at (-7,0){};
\node[circle,draw,fill = black,scale=.85]  (c) at (-7,-1){};
\node[circle,draw,fill = black,scale=.85]  (c) at (-7,-2){};
\node[circle,draw,fill = black,scale=.85]  (c) at (-7,-3){};
\node[circle,draw,fill = black,scale=.85]  (c) at (-7,-4){};
\node[circle,draw,fill = black,scale=.85]  (c) at (-7,-5){};
\node[circle,draw,fill = black,scale=.85]  (c) at (-7,-6){};
\node[circle,draw,fill = black,scale=.85]  (c) at (-7,-7){};
\node[circle,draw,fill = black,scale=.85]  (c) at (-7,-8){};
\node[circle,draw,fill = black,scale=.85]  (c) at (-7,-9){};
\node[circle,draw,fill = black,scale=.85]  (c) at (-7,-10){};
\node[circle,draw,fill = black,scale=.85]  (c) at (-7,-11){};

\node[circle,draw,fill = white,scale=.85]  (c) at (-6,5){};
\node[circle,draw,fill = white,scale=.85]  (c) at (-6,4){};
\node[circle,draw,fill = white,scale=.85]  (c) at (-6,3){};
\node[circle,draw,fill = black,scale=.85]  (c) at (-6,2){};
\node[circle,draw,fill = black,scale=.85]  (c) at (-6,1){};
\node[circle,draw,fill = black,scale=.85]  (c) at (-6,0){};
\node[circle,draw,fill = black,scale=.85]  (c) at (-6,-1){};
\node[circle,draw,fill = black,scale=.85]  (c) at (-6,-2){};
\node[circle,draw,fill = black,scale=.85]  (c) at (-6,-3){};
\node[circle,draw,fill = black,scale=.85]  (c) at (-6,-4){};
\node[circle,draw,fill = black,scale=.85]  (c) at (-6,-5){};
\node[circle,draw,fill = black,scale=.85]  (c) at (-6,-6){};
\node[circle,draw,fill = black,scale=.85]  (c) at (-6,-7){};
\node[circle,draw,fill = black,scale=.85]  (c) at (-6,-8){};
\node[circle,draw,fill = black,scale=.85]  (c) at (-6,-9){};
\node[circle,draw,fill = black,scale=.85]  (c) at (-6,-10){};
\node[circle,draw,fill = black,scale=.85]  (c) at (-6,-11){};

\node[circle,draw,fill = white,scale=.85]  (c) at (-5,5){};
\node[circle,draw,fill = white,scale=.85]  (c) at (-5,4){};
\node[circle,draw,fill = white,scale=.85]  (c) at (-5,3){};
\node[rectangle,draw, fill = white,scale=.85]    (c) at (-5,2){};
\node[circle,draw,fill = black,scale=.85]  (c) at (-5,1){};
\node[circle,draw,fill = black,scale=.85]  (c) at (-5,0){};
\node[circle,draw,fill = black,scale=.85]  (c) at (-5,-1){};
\node[circle,draw,fill = black,scale=.85]  (c) at (-5,-2){};
\node[circle,draw,fill = black,scale=.85]  (c) at (-5,-3){};
\node[circle,draw,fill = black,scale=.85]  (c) at (-5,-4){};
\node[circle,draw,fill = black,scale=.85]  (c) at (-5,-5){};
\node[circle,draw,fill = black,scale=.85]  (c) at (-5,-6){};
\node[circle,draw,fill = black,scale=.85]  (c) at (-5,-7){};
\node[circle,draw,fill = black,scale=.85]  (c) at (-5,-8){};
\node[circle,draw,fill = black,scale=.85]  (c) at (-5,-9){};
\node[circle,draw,fill = black,scale=.85]  (c) at (-5,-10){};
\node[circle,draw,fill = black,scale=.85]  (c) at (-5,-11){};

\node[circle,draw,fill = white,scale=.85]  (c) at (-4,5){};
\node[circle,draw,fill = white,scale=.85]  (c) at (-4,4){};
\node[circle,draw,fill = white,scale=.85]  (c) at (-4,3){};
\node[circle,draw,fill = white,scale=.85]  (c) at (-4,2){};
\node[circle,draw,fill = white,scale=.85]  (c) at (-4,1){};
\node[scale=.9]  (c) at (-4,0){\statcirc[white]{black}};;
\node[circle,draw,fill = black,scale=.85]  (c) at (-4,-1){};
\node[circle,draw,fill = black,scale=.85]  (c) at (-4,-2){};
\node[circle,draw,fill = black,scale=.85]  (c) at (-4,-3){};
\node[circle,draw,fill = black,scale=.85]  (c) at (-4,-4){};
\node[circle,draw,fill = black,scale=.85]  (c) at (-4,-5){};
\node[circle,draw,fill = black,scale=.85]  (c) at (-4,-6){};
\node[circle,draw,fill = black,scale=.85]  (c) at (-4,-7){};
\node[circle,draw,fill = black,scale=.85]  (c) at (-4,-8){};
\node[circle,draw,fill = black,scale=.85]  (c) at (-4,-9){};
\node[circle,draw,fill = black,scale=.85]  (c) at (-4,-10){};
\node[circle,draw,fill = black,scale=.85]  (c) at (-4,-11){};

\node[circle,draw,fill = white,scale=.85]  (c) at (-3,5){};
\node[circle,draw,fill = white,scale=.85]  (c) at (-3,4){};
\node[circle,draw,fill = white,scale=.85]  (c) at (-3,3){};
\node[circle,draw,fill = white,scale=.85]  (c) at (-3,2){};
\node[circle,draw,fill = white,scale=.85]  (c) at (-3,1){};
\node[scale=.9]  (c) at (-3,0){\statcirc[white]{black}};;
\node[scale=.9]  (c) at (-3,-1){\statcirc[white]{black}};;
\node[circle,draw,fill = black,scale=.85]  (c) at (-3,-2){};
\node[circle,draw,fill = black,scale=.85]  (c) at (-3,-3){};
\node[circle,draw,fill = black,scale=.85]  (c) at (-3,-4){};
\node[circle,draw,fill = black,scale=.85]  (c) at (-3,-5){};
\node[circle,draw,fill = black,scale=.85]  (c) at (-3,-6){};
\node[circle,draw,fill = black,scale=.85]  (c) at (-3,-7){};
\node[circle,draw,fill = black,scale=.85]  (c) at (-3,-8){};
\node[circle,draw,fill = black,scale=.85]  (c) at (-3,-9){};
\node[circle,draw,fill = black,scale=.85]  (c) at (-3,-10){};
\node[circle,draw,fill = black,scale=.85]  (c) at (-3,-11){};

\node[circle,draw,fill = white,scale=.85]  (c) at (-2,5){};
\node[circle,draw,fill = white,scale=.85]  (c) at (-2,4){};
\node[circle,draw,fill = white,scale=.85]  (c) at (-2,3){};
\node[circle,draw,fill = white,scale=.85]  (c) at (-2,2){};
\node[circle,draw,fill = white,scale=.85]  (c) at (-2,1){};
\node[circle,draw,fill = white,scale=.85]  (c) at (-2,0){};
\node[circle,draw,fill = white,scale=.85]  (c) at (-2,-1){};
\node[scale=.9]  (c) at (-2,-2){\statcirc[white]{black}};;
\node[scale=.9]  (c) at (-2,-3){\statcirc[white]{black}};;
\node[circle,draw,fill = black,scale=.85]  (c) at (-2,-4){};
\node[circle,draw,fill = black,scale=.85]  (c) at (-2,-5){};
\node[circle,draw,fill = black,scale=.85]  (c) at (-2,-6){};
\node[circle,draw,fill = black,scale=.85]  (c) at (-2,-7){};
\node[circle,draw,fill = black,scale=.85]  (c) at (-2,-8){};
\node[circle,draw,fill = black,scale=.85]  (c) at (-2,-9){};
\node[circle,draw,fill = black,scale=.85]  (c) at (-2,-10){};
\node[circle,draw,fill = black,scale=.85]  (c) at (-2,-11){};

\node[circle,draw,fill = white,scale=.85]  (c) at (-1,5){};
\node[circle,draw,fill = white,scale=.85]  (c) at (-1,4){};
\node[circle,draw,fill = white,scale=.85]  (c) at (-1,3){};
\node[circle,draw,fill = white,scale=.85]  (c) at (-1,2){};
\node[circle,draw,fill = white,scale=.85]  (c) at (-1,1){};
\node[circle,draw,fill = white,scale=.85]  (c) at (-1,0){};
\node[circle,draw,fill = white,scale=.85]  (c) at (-1,-1){};
\node[scale=.9]  (c) at (-1,-2){\statcirc[white]{black}};;
\node[scale=.9]  (c) at (-1,-3){\statcirc[white]{black}};;
\node[scale=.9]  (c) at (-1,-4){\statcirc[white]{black}};;
\node[scale=.9]  (c) at (-1,-5){\statcirc[white]{black}};;
\node[circle,draw,fill = black,scale=.85]  (c) at (-1,-6){};
\node[circle,draw,fill = black,scale=.85]  (c) at (-1,-7){};
\node[circle,draw,fill = black,scale=.85]  (c) at (-1,-8){};
\node[circle,draw,fill = black,scale=.85]  (c) at (-1,-9){};
\node[circle,draw,fill = black,scale=.85]  (c) at (-1,-10){};
\node[circle,draw,fill = black,scale=.85]  (c) at (-1,-11){};

\node[circle,draw,fill = white,scale=.85]  (c) at (0,5){};
\node[circle,draw,fill = white,scale=.85]  (c) at (0,4){};
\node[circle,draw,fill = white,scale=.85]  (c) at (0,3){};
\node[circle,draw,fill = white,scale=.85]  (c) at (0,2){};
\node[circle,draw,fill = white,scale=.85]  (c) at (0,1){};
\node[circle,draw,fill = white,scale=.85]  (c) at (0,0){};
\node[circle,draw,fill = white,scale=.85]  (c) at (0,-1){};
\node[circle,draw,fill = white,scale=.85]  (c) at (0,-2){};
\node[circle,draw,fill = white,scale=.85]  (c) at (0,-3){};
\node[scale=.9]  (c) at (0,-4){\statcirc[white]{black}};;
\node[scale=.9]  (c) at (0,-5){\statcirc[white]{black}};;
\node[scale=.9]  (c) at (0,-6){\statcirc[white]{black}};;
\node[scale=.9]  (c) at (0,-7){\statcirc[white]{black}};;
\node[circle,draw,fill = black,scale=.85]  (c) at (0,-8){};
\node[circle,draw,fill = black,scale=.85]  (c) at (0,-9){};
\node[circle,draw,fill = black,scale=.85]  (c) at (0,-10){};
\node[circle,draw,fill = black,scale=.85]  (c) at (0,-11){};

\node[circle,draw,fill = white,scale=.85]  (c) at (1,5){};
\node[circle,draw,fill = white,scale=.85]  (c) at (1,4){};
\node[circle,draw,fill = white,scale=.85]  (c) at (1,3){};
\node[circle,draw,fill = white,scale=.85]  (c) at (1,2){};
\node[circle,draw,fill = white,scale=.85]  (c) at (1,1){};
\node[circle,draw,fill = white,scale=.85]  (c) at (1,0){};
\node[circle,draw,fill = white,scale=.85]  (c) at (1,-1){};
\node[circle,draw,fill = white,scale=.85]  (c) at (1,-2){};
\node[circle,draw,fill = white,scale=.85]  (c) at (1,-3){};
\node[scale=.9]  (c) at (1,-4){\statcirc[white]{black}};;
\node[scale=.9]  (c) at (1,-5){\statcirc[white]{black}};;
\node[scale=.9]  (c) at (1,-6){\statcirc[white]{black}};;
\node[scale=.9]  (c) at (1,-7){\statcirc[white]{black}};;
\node[scale=.9]  (c) at (1,-8){\statcirc[white]{black}};;
\node[scale=.9]  (c) at (1,-9){\statcirc[white]{black}};;
\node[scale=.9]  (c) at (1,-10){\statcirc[white]{black}};;
\node[circle,draw,fill = black,scale=.85]  (c) at (1,-11){};

\node[circle,draw,fill = white,scale=.85]  (c) at (2,5){};
\node[circle,draw,fill = white,scale=.85]  (c) at (2,4){};
\node[circle,draw,fill = white,scale=.85]  (c) at (2,3){};
\node[circle,draw,fill = white,scale=.85]  (c) at (2,2){};
\node[circle,draw,fill = white,scale=.85]  (c) at (2,1){};
\node[circle,draw,fill = white,scale=.85]  (c) at (2,0){};
\node[circle,draw,fill = white,scale=.85]  (c) at (2,-1){};
\node[circle,draw,fill = white,scale=.85]  (c) at (2,-2){};
\node[circle,draw,fill = white,scale=.85]  (c) at (2,-3){};
\node[circle,draw,fill = white,scale=.85]  (c) at (2,-4){};
\node[circle,draw,fill = white,scale=.85]  (c) at (2,-5){};
\node[scale=.9]  (c) at (2,-6){\statcirc[white]{black}};;
\node[scale=.9]  (c) at (2,-7){\statcirc[white]{black}};;
\node[scale=.9]  (c) at (2,-8){\statcirc[white]{black}};;
\node[scale=.9]  (c) at (2,-9){\statcirc[white]{black}};;
\node[scale=.9]  (c) at (2,-10){\statcirc[white]{black}};;
\node[scale=.9]  (c) at (2,-11){\statcirc[white]{black}};;

\node[circle,draw,fill = white,scale=.85]  (c) at (3,5){};
\node[circle,draw,fill = white,scale=.85]  (c) at (3,4){};
\node[circle,draw,fill = white,scale=.85]  (c) at (3,3){};
\node[circle,draw,fill = white,scale=.85]  (c) at (3,2){};
\node[circle,draw,fill = white,scale=.85]  (c) at (3,1){};
\node[circle,draw,fill = white,scale=.85]  (c) at (3,0){};
\node[circle,draw,fill = white,scale=.85]  (c) at (3,-1){};
\node[circle,draw,fill = white,scale=.85]  (c) at (3,-2){};
\node[circle,draw,fill = white,scale=.85]  (c) at (3,-3){};
\node[circle,draw,fill = white,scale=.85]  (c) at (3,-4){};
\node[circle,draw,fill = white,scale=.85]  (c) at (3,-5){};
\node[scale=.9]  (c) at (3,-6){\statcirc[white]{black}};;
\node[scale=.9]  (c) at (3,-7){\statcirc[white]{black}};;
\node[scale=.9]  (c) at (3,-8){\statcirc[white]{black}};;
\node[scale=.9]  (c) at (3,-9){\statcirc[white]{black}};;
\node[scale=.9]  (c) at (3,-10){\statcirc[white]{black}};;
\node[scale=.9]  (c) at (3,-11){\statcirc[white]{black}};;

\node[circle,draw,fill = white,scale=.85]  (c) at (4,5){};
\node[circle,draw,fill = white,scale=.85]  (c) at (4,4){};
\node[circle,draw,fill = white,scale=.85]  (c) at (4,3){};
\node[circle,draw,fill = white,scale=.85]  (c) at (4,2){};
\node[circle,draw,fill = white,scale=.85]  (c) at (4,1){};
\node[circle,draw,fill = white,scale=.85]  (c) at (4,0){};
\node[circle,draw,fill = white,scale=.85]  (c) at (4,-1){};
\node[circle,draw,fill = white,scale=.85]  (c) at (4,-2){};
\node[circle,draw,fill = white,scale=.85]  (c) at (4,-3){};
\node[circle,draw,fill = white,scale=.85]  (c) at (4,-4){};
\node[circle,draw,fill = white,scale=.85]  (c) at (4,-5){};
\node[circle,draw,fill = white,scale=.85]  (c) at (4,-6){};
\node[circle,draw,fill = white,scale=.85]  (c) at (4,-7){};
\node[scale=.9]  (c) at (4,-8){\statcirc[white]{black}};;
\node[scale=.9]  (c) at (4,-9){\statcirc[white]{black}};;
\node[scale=.9]  (c) at (4,-10){\statcirc[white]{black}};;
\node[scale=.9]  (c) at (4,-11){\statcirc[white]{black}};;

\node[circle,draw,fill = white,scale=.85]  (c) at (5,5){};
\node[circle,draw,fill = white,scale=.85]  (c) at (5,4){};
\node[circle,draw,fill = white,scale=.85]  (c) at (5,3){};
\node[circle,draw,fill = white,scale=.85]  (c) at (5,2){};
\node[circle,draw,fill = white,scale=.85]  (c) at (5,1){};
\node[circle,draw,fill = white,scale=.85]  (c) at (5,0){};
\node[circle,draw,fill = white,scale=.85]  (c) at (5,-1){};
\node[circle,draw,fill = white,scale=.85]  (c) at (5,-2){};
\node[circle,draw,fill = white,scale=.85]  (c) at (5,-3){};
\node[circle,draw,fill = white,scale=.85]  (c) at (5,-4){};
\node[circle,draw,fill = white,scale=.85]  (c) at (5,-5){};
\node[circle,draw,fill = white,scale=.85]  (c) at (5,-6){};
\node[circle,draw,fill = white,scale=.85]  (c) at (5,-7){};
\node[scale=.9]  (c) at (5,-8){\statcirc[white]{black}};;
\node[scale=.9]  (c) at (5,-9){\statcirc[white]{black}};;
\node[scale=.9]  (c) at (5,-10){\statcirc[white]{black}};;
\node[scale=.9]  (c) at (5,-11){\statcirc[white]{black}};;

\node[circle,draw,fill = white,scale=.85]  (c) at (6,5){};
\node[circle,draw,fill = white,scale=.85]  (c) at (6,4){};
\node[circle,draw,fill = white,scale=.85]  (c) at (6,3){};
\node[circle,draw,fill = white,scale=.85]  (c) at (6,2){};
\node[circle,draw,fill = white,scale=.85]  (c) at (6,1){};
\node[circle,draw,fill = white,scale=.85]  (c) at (6,0){};
\node[circle,draw,fill = white,scale=.85]  (c) at (6,-1){};
\node[circle,draw,fill = white,scale=.85]  (c) at (6,-2){};
\node[circle,draw,fill = white,scale=.85]  (c) at (6,-3){};
\node[circle,draw,fill = white,scale=.85]  (c) at (6,-4){};
\node[circle,draw,fill = white,scale=.85]  (c) at (6,-5){};
\node[circle,draw,fill = white,scale=.85]  (c) at (6,-6){};
\node[circle,draw,fill = white,scale=.85]  (c) at (6,-7){};
\node[circle,draw,fill = white,scale=.85]  (c) at (6,-8){};
\node[circle,draw,fill = white,scale=.85]  (c) at (6,-9){};
\node[scale=.9]  (c) at (6,-10){\statcirc[white]{black}};;
\node[scale=.9]  (c) at (6,-11){\statcirc[white]{black}};;

\node[circle,draw,fill = white,scale=.85]  (c) at (7,5){};
\node[circle,draw,fill = white,scale=.85]  (c) at (7,4){};
\node[circle,draw,fill = white,scale=.85]  (c) at (7,3){};
\node[circle,draw,fill = white,scale=.85]  (c) at (7,2){};
\node[circle,draw,fill = white,scale=.85]  (c) at (7,1){};
\node[circle,draw,fill = white,scale=.85]  (c) at (7,0){};
\node[circle,draw,fill = white,scale=.85]  (c) at (7,-1){};
\node[circle,draw,fill = white,scale=.85]  (c) at (7,-2){};
\node[circle,draw,fill = white,scale=.85]  (c) at (7,-3){};
\node[circle,draw,fill = white,scale=.85]  (c) at (7,-4){};
\node[circle,draw,fill = white,scale=.85]  (c) at (7,-5){};
\node[circle,draw,fill = white,scale=.85]  (c) at (7,-6){};
\node[circle,draw,fill = white,scale=.85]  (c) at (7,-7){};
\node[circle,draw,fill = white,scale=.85]  (c) at (7,-8){};
\node[circle,draw,fill = white,scale=.85]  (c) at (7,-9){};
\node[scale=.9]  (c) at (7,-10){\statcirc[white]{black}};;
\node[scale=.9]  (c) at (7,-11){\statcirc[white]{black}};;

\node[circle,draw, fill = white] (c) at (8,3){};
\node at (9.6,3) {  Yes, \ref{thm:koszulThm}}; 

\node (c) at (8,-1){\statcirc[white]{black}};
\node at (9.85,-1) {Unknown}; 

\node[rectangle,draw, fill = black]  (c) at (8,0){};
\node at (9.6,0) {No, \ref{prop:exampleNotKoszul}}; 

\node[rectangle,draw, fill = white]  (c) at (8,2){};
\node at (9.7,2) {Yes, \ref{prop:exampleKoszul}}; 

\node[circle,draw, fill = black] (c) at (8,1){};
\node at (9.65,1) {No, \ref{thm:notKoszul}}; 

\end{tikzpicture}

\end{minipage}
\caption{The Koszul property for the coordinate ring of $m$ generic lines in $\mathbb{P}^n$}
\end{figure}

\newpage 
\section*{Acknowledgements}

The author would like to thank Jason McCullough; you have been a wonderful advisor and mentor for me. This research was partially supported by NSF grant DMS-1900792.

\bibliography{MyBib}{}
\bibliographystyle{plain.bst}

\end{document}